\documentclass[12pt,reqno]{amsart}

\usepackage[all,cmtip]{xy}
\usepackage{tikz-cd}
\usepackage{enumitem}
\usepackage{parskip}
\usepackage{nicematrix}
\usepackage{amssymb}

\usepackage{hyperref}

\numberwithin{equation}{section}

\newtheorem{theorem}{Theorem}[section]
\newtheorem{corollary}[theorem]{Corollary}
\newtheorem{lemma}[theorem]{Lemma}
\newtheorem{proposition}[theorem]{Proposition}

\theoremstyle{definition}
\newtheorem{definition}[theorem]{Definition}
\newtheorem{example}[theorem]{Example}

\newtheorem{remark}[theorem]{Remark}

\newcommand{\GmodX}[3]{#1 \backslash\mkern-6mu\backslash #2^{ss}_{#1}(#3)}
\newcommand{\opsgmodx}[1]{\GmodX{\lambda_{s}}{X(w_{s,r})}{\mathcal{L}(#1\omega_{r})}}

\begin{document}	
	\title{ GIT quotient of Schubert varieties modulo one dimensional torus}	
	\author{Arkadev Ghosh }
	\address{Chennai Mathematical Institute, Plot H1, SIPCOT IT Park, Siruseri, Kelambakkam, 603103, India.}
	\email{arkadev@cmi.ac.in}	
	\author{S. S. Kannan}
	\address{Chennai Mathematical Institute, Plot H1, SIPCOT IT Park, Siruseri, Kelambakkam, 603103, India.}	
	\email{kannan@cmi.ac.in}	
	\subjclass[2010]{14M15, 14L35, 14F25}	
	\keywords{Semi-stable points, minuscule fundamental weights, Schubert variety}				
	\maketitle 		
	\begin{center}
		To the memory of C. S. Seshadri
	\end{center}
	\begin{abstract}
		Let $G$ be a simple algebraic group of adjoint type of rank $n$ over $\mathbb{C}$. Let $T$ be a maximal torus of $G$, and $B$ be a Borel subgroup of $G$ containing $T$. Let $W=N_{G}(T)/T$ be the Weyl group of $G$. Let $S=\{\alpha_{1},\ldots,\alpha_{n}\}$ be the set of simple roots of $G$ relative to $(B,T)$. Let $\lambda_{s}$ be the one parameter subgroup of $T$ dual to $\alpha_{s}$. In this paper, we give a criterion for Schubert varieties admitting semistable points for the $\lambda_{s}$-linearized line bundles $\mathcal{L}(\chi)$ associated to every dominant character $\chi$ of $T$. If $\omega_{r}$ is a minuscule fundamental weight and $m\omega_{r}\in X(T)$, then we prove that there is a unique minimal dimensional Schubert variety $X(w_{s,r})$ in $G/P_{S\setminus\{\alpha_{r}\}}$ such that $X(w_{s,r})^{ss}_{\lambda_{s}}(\mathcal{L}(m\omega_{r}))\neq \phi$. Further, we prove that if $G=PSL(n,\mathbb{C})$, and $n\nmid rs$, $m=\frac{n}{(rs,n)}$, and $p=\lfloor\frac{rs}{n}\rfloor$ then $\GmodX{\lambda_{s}}{X(w_{s,r})}{\mathcal{L}(m\omega_{r})}\simeq (\mathbb{P}(M(s-p,r-p)),\mathbb{O}(a))$ for some $a\in\mathbb{N}$.
	\end{abstract}		
	\section{Introduction}
	Let $G$ be a simple  algebraic group of adjoint type of rank $n$ over $\mathbb{C}$.  Let $T$ be a maximal torus of $G$ and $B$ be a Borel subgroup of $G$ containing $T$. Let $\pi: \hat{G}\rightarrow G$ be the simply connected covering. Let $\hat{T}=\pi^{-1}(T)$, $\hat{B}=\pi^{-1}(B)$. Let $W\,=\,N_{G}(T)/T$ denote the 
	Weyl group of $G$ with respect to $T$.  The Lie algebra of $G$ will be denoted by  $\mathfrak{g}$.
	Let $\mathfrak{b}\subset \mathfrak{g}$ and $\mathfrak{h}\subset \mathfrak{b}$ be the Lie algebras of $B$ and
	$T$ respectively. The set of roots of $G$ with respect to $T$ will be denoted by $R$. Let $R^{+}\subseteq R$ be the set of positive roots with respect to $B$. Let $S=\{\alpha_{1},\ldots,\alpha_{n}\}\subseteq R^{+}$ be the set of simple roots with respect to $B$. Let $\{\omega_{r}:1\leq r\leq n\}$ be the fundamental weights associated to $S$. \\
	In \cite{Hausmann}, Hausmann and Knutson identified the GIT quotient of the Grassmannian $G_{2,n}$ by $T$ with the moduli space of polygons in $\mathbb{R}^{3}$. Also, they showed that GIT quotient of $G_{2,n}$ by $T$ can be realized as the GIT quotient of an $n$-fold product of projective lines by the diagonal action of $PSL(2,\mathbb{C})$. More generally, using Gel'fand - Macpherson correspondence GIT quotient of $G_{r.n}$ by $T$ can be identified with the GIT  quotient of $(\mathbb{P}^{r-1})^{n}$ by the diagonal action of $PSL(r,\mathbb{C})$.   \\	
	In \cite{Skorobogatov}, Skorobogatov gave a combinatorial description of set of semistable (resp. stable) points in $G_{r,n}$ with respect to $T$-linearized line bundle $\mathcal{L}(\omega_{r})$. As a corollary he showed that when $r$ and $n$ are coprime, semistability is same as stability.
	In \cite{Torus quptient I} and \cite{Torus quotient II}, second named author characterized the parabolic subgroups $P$  of a simple algebraic group $G$ containing $B$ for which there exists an ample line bundle $\mathcal{L}$ on $G/P$ such that $(G/P)^{ss}_{T}(\mathcal{L})=(G/P)^{s}_{T}(\mathcal{L})$. He also independently proved that $(G_{r,n})_{T}^{s}(\mathcal{L}(\omega_{r}))=(G_{r,n})_{T}^{ss}(\mathcal{L}(\omega_{r}))$ if and only if $r$ and $n$ are coprime.
	In \cite{s kannan}, second named author and Sardar studied torus quotient of Schubert varieties in $G_{r,n}$.
	They showed that $G_{r,n}$ has a unique minimal dimensional Schubert variety $X(w)$ admitting
	semistable points with respect to the $T$-lineararized bundle $\mathcal{L}(\omega_{r})$, and gave a combinatorial
	characterization of the Weyl group element $w$. In \cite{KP}, second named author and Pattanayak
	extended the results to the case when $G$ is of type $B$, $C$, or $D$ and when $P$ is a maximal
	parabolic subgroup of $G$.\\		
	In view of the study of all above authors, it is a natural problem to study the GIT quotients of $G_{r,n}$ with respect to $H$-linearized line bundle $\mathcal{L}(\omega_{r})$ for any subtorus $H$ of $T$. Let $X(T)$ (respectively, $Y(T)$) be the group of characters (respectively, group of one parameter subgroups of $T$) and let $\langle \cdot,\cdot\rangle: X(T)\times Y(T)\rightarrow \mathbb{Z}$ be the canonical pairing.
	For $1\leq i\leq n$, let $\lambda_{i}\in X(T)$ be the one parameter subgroup of $T$ dual to $\alpha_{i}$. That is, $\langle\alpha_{j},\lambda_{i}\rangle=\delta_{i,j}$ for all $1\leq i,j\leq n$.\\
	In \cite{Torus quptient I}, second named author used one parameter subgroup $\lambda_{s}$ to prove that $(G_{r,n})^{ss}_{T}(\mathcal{L}(\omega_{r}))=(G_{r,n})^{s}_{T}(\mathcal{L}(\omega_{r}))$ if and only if $(r,n)=1$ (see \cite[Theorem 3.3]{Torus quptient I}). In  \cite{Torus quotient II}, author used $\lambda_{s}$ in the computation to prove that, $G$ is of type different from $A_{n}$, if $Q$ is a parabolic subgroup of $G$ containing $B$, then there exists an ample line bundle $\mathcal{L}$ such that $(G/Q)^{ss}_{T}(\mathcal{L})=(G/Q)^{s}_{T}(\mathcal{L})$ if and only if $Q=B$. In  \cite{KP}, authors used one parameter subgroups $\lambda_{s}$ to prove that there is a unique minimal dimensional Schubert variety $X(w)$ in a minuscule flag variety admitting semistable points with respect to the $T$-linearized line bundle $\mathcal{L}(\omega_{r})$. In \cite{s kannan},  authors used $\lambda_{s}$ to describe all minimal dimensional Schubert variety admitting semistable point with respect to $T$-linearized line bundle $\mathcal{L}(\omega_{r})$ on $G/P_{S\setminus\{\alpha_{r}\}}$ for $G$ is of type $B$, $C$, $D$ and all maximal parabolic subgroups $P_{S\setminus\{\alpha_{r}\}}$.
	
	 So, it is a natural problem to study $\GmodX{\lambda_{s}}{(G/P)}{\mathcal{L}}$ for any $G,P,\mathcal{L}$ as above. This paper is an attempt to solve the problems for the case of the one parameter subgroup $\lambda_{s}$ that are analogous to the case of $T$. For instance, it is showed in \cite{Torus quotient II} that for any semisimple algebraic group $G$ and a parabolic subgroup $P$ of $G$ containing $B$, and an ample line bundle $\mathcal{L}(\chi)$ on $G/P$ associated to a dominant character $\chi$ of $T$, $(G/P)^{ss}_{T}(\mathcal{L}(\chi))=(G/P)^{s}_{T}(\mathcal{L}(\chi))$ if and only if $\langle w(\chi),\lambda_{s}\rangle\neq 0$ for all $w\in W$ and all $1\leq s\leq n$ (see \cite[Lemma 4.1, p.38]{Torus quotient II}). In this paper, for any $1\leq s\leq n$, we prove that $(G/P)^{ss}_{\lambda_{s}}(\mathcal{L}(\chi))=(G/P)^{s}_{\lambda_{s}}(\mathcal{L}(\chi))$ if and only if $\langle w(\chi),\lambda_{s}\rangle\neq 0$ for all $w\in W$ (see Corollary \ref{ss=s for G/P I}). In \cite{s kannan}, it is proved that for $G=PSL(n,\mathbb{C})$ and $1\leq r\leq n-1$, there is a unique minimal dimensional Schuert variety in $G_{r,n}$ admitting semistable point for the $T$-linearized line bundle $\mathcal{L}(n\omega_{r})$ (see \cite[Lemma 2.7, p.88]{s kannan}).
	In view of this result, there is a natural question : for any $1\leq s\leq n$ and for any minuscule fundamental weight $\omega_{r}$, is there a unique minimal dimensional Schubert variety $X(w_{s,r})\subseteq G/P_{S\setminus \{\alpha_{r}\}}$ admitting semistable point for $\lambda_{s}$? In this paper, for any minuscule fundamental weight $\omega_{r}$ and $m\in \mathbb{N}$ is such that $m\omega_{r}$ is in the root lattice, we prove that there is a unique minimal dimensional Schubert variety $X(w_{s,r})$ admitting semistable point for the $\lambda_{s}$-linearized line bundle $\mathcal{L}(m\omega_{r})$. It is proved in \cite{KP} that for any $G,P,\mathcal{L}(\chi)$ as above and a Schubert variety $X(w)\subseteq G/P$, $X(w)^{ss}_{T}(\mathcal{L}(\chi))\neq\phi$ if and only if $w(\chi)\leq 0$. In this paper, we prove the following :   $X(w)^{ss}_{\lambda_{s}}(\mathcal{L}(\chi))\neq\phi$ if and only if $\langle w(\chi),\lambda_{s}\rangle\leq0$ (see Lemma \ref{main lemma}).	
	
	Let $\omega_{r}$ be a minuscule fundamental weight.
	Let $m\in \mathbb{N}$ be a positive integer such that $m\omega_{r}$ is in the root lattice. In this paper we prove the following results.
	\begin{lemma} $($see Corollary \ref{qt is point, special case} $)$	If $X(w_{s,r})^{ss}_{\lambda_{s}}(\mathcal{L}(m\omega_{r}))\neq X(w_{s,r})^{s}_{\lambda_{s}}(\mathcal{L}(m\omega_{r}))$, then the quotient $\opsgmodx{m}$ is a point.
	\end{lemma}
	\begin{proposition}$($see Proposition \ref{quotient is projective space}$)$
		If $\alpha_{s}$ is cominuscule, and $X(w_{s,r})^{ss}_{\lambda_{s}}(\mathcal{L}(m\omega_{r}))=X(w_{s,r})^{s}_{\lambda_{s}}(\mathcal{L}(m\omega_{r}))$, then $\opsgmodx{m}$
		is a projective space.
	\end{proposition}	
	Since $C_{G}(T)=T$ and $N_{G}(T)/T$ is finite, there is no natural action of a positive dimensional connected reductive group arising from a subgroup of $G$ on $\displaystyle\bigoplus_{d\in \mathbb{Z}_{\geq 0}}H^{0}(G/P,\mathcal{L}(d\chi))^{T}$. The advantage for $\lambda_{s}$ is that the reductive group $C_{G}(\lambda_{s})$ acts on $\displaystyle\bigoplus_{d\in \mathbb{Z}_{\geq 0}}H^{0}(G/P,\mathcal{L}(d\chi))^{\lambda_{s}}$ naturally. Consider the induced action of $C_{G}(\lambda_{s})$ on $\GmodX{\lambda_{s}}{(G/P)}{\mathcal{L}(\chi)}$. Therefore, it is an interesting problem to study $\GmodX{\lambda_{s}}{(G/P)}{\mathcal{L}(\chi)}$ as a $C_{G}(\lambda_{s})$-variety. For the special case of $\hat{G}=SL(n,\mathbb{C})$, and for $1\leq r,s\leq n-1$, let $w=w_{s,r}$ and $P_{w}$ be the stabilizer of $X(w)$ in $\hat{G}$.  Let $c=\frac{rs}{(rs,n)}$,  $m=\frac{n}{(rs,n)}$, $p=\lfloor\frac{rs}{n}\rfloor$ and $a=c-mp$.  Assume that $X(w)^{ss}_{\lambda_{s}}(\mathcal{L}(m\omega_{r}))=X(w)^{s}_{\lambda_{s}}(\mathcal{L}(m\omega_{r}))$. In this case, the semisimple part of the Levi subgroup of $C_{P_{w}}(n\cdot\lambda_{s})$ is $SL(s-p,\mathbb{C})\times SL(r-p,\mathbb{C})$. So, the natural problem is to study 
	\begin{enumerate}[label=(\roman*)]
		\item the $SL(s-p,\mathbb{C})\times SL(r-p,\mathbb{C})$-module $H^{0}(X(w),\mathcal{L}(dm\omega_{r}))^{\lambda_{s}}$ for any  $d\in \mathbb{Z}_{\geq 0}$ 
		\item to describe $\opsgmodx{m}$ as a $SL(s-p,\mathbb{C})\times SL(r-p,\mathbb{C})$-variety.
	\end{enumerate}	In this direction, we prove the following (see Theorem \ref{git quotient type An}).
	\begin{theorem}
		$\opsgmodx{m}=(\mathbb{P}(M(s-p,r-p)),\mathbb{O}(a))$, where $M(s-p,r-p)$ is the vector space of all $s-p\times r-p$ matrices with entries in $\mathbb{C}$.\label{main result}
	\end{theorem}
		\begin{subsection}{Organization}
		This paper is organized into the following sections. In section 2, we recall some notations and preliminaries on algebraic groups, Lie 
		algebras, standard monomial theory, and Geometric Invariant Theory. In section 3, we prove that a Schubert variety $X(w)$ in $G/P$ admits  semistable point for the action of one parameter subgroup $\lambda_{s}$ if and only if $\langle w(\chi),\lambda_{s}\rangle\leq 0$, where $\chi$ is a dominant character of $P$ (see Lemma \ref{main lemma}). In section 4, we prove that there is unique minimal dimensional Schubert variety $X(w_{s,r})$ admitting semistable points in minuscule $G/P_{S\setminus\{\alpha_{r}\}}$, and is described in section 5. In section 6, we prove that if $\omega_{r}$ is minuscule and $\lambda_{s}$ is cominuscule, then $\opsgmodx{m}$ is either a point or projective space. In section 7, we prove Theorem \ref{main result}.
	\end{subsection} 
	\section{NOTATION AND PRELIMINARIES}
	In this section, we set up some notations and preliminaries. We refer to  \cite{BK05},
	\cite{Hum1}, \cite{Hum2}, \cite{Jan}, \cite{GIT}, and \cite{NEWSTEAD}.
	The simple reflection in $W$ corresponding to $\alpha_{i}$ is denoted by $s_{i}$. For a subset $I$ of $S$, we denote the parabolic subgroup of $G$ generated by $B$ and $\{n_{\alpha}:\alpha\in I\}$ by $P_{I}$, where $n_{\alpha}$ is a representative of $s_{\alpha}$ in $N_{G}(T)$.
	Note that all the standard parabolic subgroups of $G$ containing $B$ are of the form $P_{I}$ for some $I\subseteq S$. Let $W_{I}$ be the subgroup of $W$ generated by $\{s_{\alpha}:\alpha\in I\}$. We note that $W_{I}$ is the Weyl group of $P_{I}$. For $I\subseteq S$, $W^{I}\ =\{w\in W | w(\alpha)\in R^{+},\ $for all $\alpha\in I\}$ is the set of minimal coset representatives of elements of $W/W_{I}$. Further there is a natural order on $W^{I}$, namely the restriction of the Bruhat order on $W$. For $w\in W$, we define $R^{+}(w^{-1})=\{\beta\in R^{+} | w^{-1}(\beta)\in R^{-}\}$. The length of any $w\,\in\, W$ is denoted by $\ell(w)$. Let $w_{0}\in W$ be the longest element of $W$. 	Let $n_{0}\in N_{G}(T)$ be a representative of $w_{0}$. Let $B^{-}:=n_{0}Bn_{0}^{-1}$ be the Borel subgroup of $G$ opposite to $B$ determined by $T$.\\
	
	Let $X(T)$ (respectively, $Y(T))$ denote the group of all characters (respectively, one-parametr subgroups) of $T$. Let $E_{1}:=\ X(T)\otimes \mathbb{R}$, and $E_{2}=\ Y(T)\otimes \mathbb{R}$.\\ Let $\langle.,.\rangle :\ E_{1}\times E_{2}\mapsto\mathbb{R}$ be the canonical non-degenerate form. Let $\overline{C(B)}:=\{\lambda\in E_{2}|\ \langle \alpha,\lambda\rangle\geq0$,\ for all $\alpha\in R^{+}\}$. 
	
	We have $X(T)\otimes \mathbb{R}=Hom_{\mathbb{R}}(\mathfrak{h}_{\mathbb{R}},\mathbb{R})$, the dual of the real form of $\mathfrak{h}$. The positive definite $W$-invariant form on $Hom_{\mathbb{R}}(\mathfrak{h}_{\mathbb{R}},\mathbb{R})$ induced by the Killing form on $\mathfrak{g}$ is denoted by $(-,-)$. For any $\mu\in X(T)\otimes\mathbb{R}$ and $\alpha\in R$, denote \[\langle \mu,\alpha\rangle=\frac{2(\mu,\alpha)}{(\alpha,\alpha)}\]
	There is a natual partial order $\le$ on $X(T)$ defined by $\psi\leq \chi$ if and only if $\chi-\psi$ is a nonnegative integral linear combination of simple roots.
	Let $u_{\alpha}:\mathbb{C}\longrightarrow U_{\alpha}$ be the isomorphism such that $tu_{\alpha}(a)t^{-1}=u_{\alpha}(\alpha(t)a)$, for all $t\in T$, $a\in \mathbb{C}$. \\
	A simple root $\alpha_{i}\in S$ is said to be cominuscule if the coefficient of $\alpha_{i}$ in the expression of highest root is $1$. We call $\lambda_{i}$ to be cominuscule, if $\alpha_{i}$ is cominuscule simple root.\\
	A fundamental weight $\omega$ is said to be minuscule if $\omega$ satisfies $\langle \omega, \beta  \rangle\le  1$ for all $\beta \in R^{+}.$\\
	If $\omega_{i}$ is a minuscule fundamental weight corresponding to the simple root $\alpha_{i},$ then the standard parabolic subgroup $P_{S\setminus\{\alpha_{i}\}}$ of $G$ corresponding to the subset $S\setminus\{\alpha_{i}\}$ of $S$ is called minuscule maximal parabolic subgroup of $G.$\\
	For minuscule fundamental weight $\omega_{i},$ the elements of $W^{S\setminus \{\alpha_{i}\}}$ are called minuscule Weyl group elements.\\
	For notations and results on Standard monomial theory, we refer to \cite{Seshadri}.\\
	We recall the definition of the Hilbert-Mumford numerical function and definition of the semistable and stable  points from \cite{GIT} (also see \cite{NEWSTEAD}).\\
	Let $G$ be a reductive group acting on a projective variety $X$. Let $\lambda$ be a one-parameter subgroup of $G$. Let $\mathcal{L}$ be a $G$-linearized very ample line bundle on $X$.
	\begin{enumerate}
		\item 
		Let $x\in \mathbb{P}(H^{0}(X,\mathcal{L})^{*})$ and $\hat{x}$ be a point in the cone $\hat{X}$ over X which lies on $x$.  Let $\{v_{i}: 1\leq i\leq k\}$ be a basis of $H^{0}(X,\mathcal{L})^{*}$ such that $\lambda(t)\cdot v_{i}=t^{m_{i}}v_{i}$, for $1\leq i\leq k$. Write $\hat{x}=\displaystyle\sum_{i=1}^{k}c_{i}v_{i}$. Then the Hilbert-Mumford numerical function is defined by $\mu^{\mathcal{L}}(x,\lambda):= -\displaystyle \min_{i}\{m_{i}| c_{i}\neq 0\}$.
		\label{definition of HM weight}
		\label{HM criterion}
		\item 
		\begin{enumerate}
			\item The set of semistable points is defined as\\ $X^{ss}_{G}(\mathcal{L})=\{x\in X|\ \exists \  s \ \in H^{0}(X,\mathcal{L}^{\otimes n})^{G}$ for some positive integer $n$ such that $s(x)\neq 0\}$.
			\label{semistability def}
			\item The set of stable points is defined as \\
			$X^{s}_{G}(\mathcal{L})=\{x\in X^{ss}_{G}(\mathcal{L})| $ the orbit $G\cdot x$ is closed in $X^{ss}_{G}(\mathcal{L})$ and the stabilizer $G_{x}$ of $x$ in $G$ is finite$\}$.
		\end{enumerate}
	\end{enumerate}
	We recall Hilbert-Mumford criterion from 
	\begin{theorem}$($see \cite[Theorem 2.1]{GIT}$)$ Let $x\in X$. Then 
		\begin{enumerate}
			\item x is semi-stable if and only if $\mu^{\mathcal{L}}(x,\lambda)\geq 0$ for all one parameter subgroup $\lambda$ of $G$.
			\item x is stable if and only if $\mu^{\mathcal{L}}(x,\lambda)> 0$ for all non trivial one parameter subgroup $\lambda$ of $G$. 
		\end{enumerate}
		\label{HM theprem for G}
	\end{theorem}
	The following special case is stated separately for future reference.
	\begin{corollary} Let $G,X, \mathcal{L}$ be as above. Let $\lambda :G_{m}\rightarrow G$ be an one parameter subgroup. Then, we have 
		\begin{enumerate}
			\item $x$ is semi-stable if and only if both $\mu^{\mathcal{L}}(x,\lambda)$ and $\mu^{\mathcal{L}}(x,-\lambda)$ are non-negative.
			\item $x$ is stable if and only if both $\mu^{\mathcal{L}}(x,\lambda)$ and $\mu^{\mathcal{L}}(x,-\lambda)$ are positive.
		\end{enumerate}
		\label{HM theorem}
	\end{corollary}
	Let $G, T, B,$ $\overline{C(B)}$ be as above.
	Let $\chi=\displaystyle\sum_{i=1}^{n}m_{i}\omega_{i}$ be a non trivial dominant character of $T$. Let $J=\{\alpha_{i}\in S:m_{i}=0\}$. Let $P=P_{J}$, and $w\in W^{J}$, $b\in B$.\\
	For $w\in W^{J}$, let $X(w):=\overline{BwP/P}$ denote the Schubert variety in $G/P$ corresponding to $w$.\\
		Here, we recall a Lemma  due to C. S. Seshadri. This will be used for computing semistable points.
	\begin{lemma} $($see \cite[Lemma 5.1]{SeS 2}$)$
		Let $x=bwP/P$. 
		Let $\lambda\in \overline{C(B)}$ be a one-parameter subgroup. Then we have \[\mu^{\mathcal{L}(\chi)}(x,\lambda)=-\langle w(\chi),\lambda\rangle.\]
		(The sign here is negative because we are using left action of P on G/P while in \cite[Lemma 5.1]{SeS 2}  the action is on the right)
		\label{CSS}
	\end{lemma}
	Following variation of the above Lemma follows from   \cite[Lemma 5.1]{SeS 2} by imitating the proof for $B^{-}$.
	\begin{lemma}
		Let $G,T,B,\chi,J, P$ and $\overline{C(B)}$ be as above. Let $w\in W^{J}$, $x\in B^{-}wP/P$. Then for every $\lambda\in \overline{C(B)}$, we have \[\mu^{\mathcal{L}(\chi)}(x,-\lambda)=\langle w(\chi),\lambda\rangle\]
		\label{variation of Seshadri's lemma}
	\end{lemma}
	\section{Semistability criterion}
	In this section, we give criterion for Schubert varities for which semistable points for the action of $\lambda_{s}$ is non empty. We also prove a criterion for semistable = stable. Let $\chi=\displaystyle\sum_{i=1}^{n}m_{i}\omega_{i}$ be a non trivial dominant character of $T$. Let $J$, $P$, $w\in W^{J}$ be as above.
	\begin{lemma} For $1\leq s\leq n$,  $X(w)^{ss}_{\lambda_{s}}(\mathcal{L}(\chi))\neq \phi$ if and only if $\langle w(\chi),\lambda_{s}\rangle\leq 0$. 
		\label{main lemma}
	\end{lemma}
	\begin{proof}
		Write $w(\chi)=\displaystyle\sum_{i=1}^{n}a_{i}\alpha_{i}$ with $a_{i}\in \mathbb{Z}$ for all $1\leq j\leq n$. Then $\langle w(\chi),\lambda_{s}\rangle=a_{s}$. We prove that $X(w)^{ss}_{\lambda_{s}}(\mathcal{L}(\chi))\neq \phi$ if and only if $a_{s}\leq 0$.\\
		(:$\Rightarrow$) Assume that $X(w)^{ss}_{\lambda_{s}}(\mathcal{L}(\chi))\neq \phi$. Since $X(w)^{ss}_{\lambda_{s}}(\mathcal{L}(\chi))$ is non empty open subset of $X(w)$ and $X(w)$ is irreducible, we have $X(w)^{ss}_{\lambda_{s}}(\mathcal{L}(\chi))\cap BwP/P \neq\phi $. So, let $x\in X(w)^{ss}_{\lambda_{s}}(\mathcal{L}(\chi))\cap BwP/P$. Then by Corollary \ref{HM theorem}, we have 
		\begin{equation}
			\mu^{\mathcal{L}(\chi)}(x,\lambda_{s})\geq 0
		\end{equation}
		By Lemma \ref{CSS}, we have
		\begin{equation}
			\mu^{\mathcal{L}(\chi)}(x,\lambda_{s})=-\langle w(\chi),\lambda_{s}\rangle=-a_{s}
		\end{equation}
		Therefore, from $(3.1)$ and $(3.2)$, we have $a_{s}\leq 0$.\\
		($\Leftarrow:$) Conversely, assume that 
		$a_{s}\leq 0$. Then, using Lemma \ref{CSS} for every $x\in BwP/P$, we have \begin{equation}
			\mu^{\mathcal{L}(\chi)}(x,\lambda_{s})=-\langle w(\chi),\lambda_{s}\rangle =-a_{s}\geq 0
		\end{equation} 
		We now show that there is an element $x\in BwP/P$ such that $\mu^{\mathcal{L}(\chi)}(x,-\lambda_{s})\geq 0$.
		Let $\widehat{X(w)}\subseteq H^{0}(G/P,\mathcal{L}(\chi))^{*}$ be the cone over $X(w)$. Note that the $G$-module $H^{0}(G/P,\mathcal{L}(\chi))^{*}$  is an irreducuble  $G$-module with highest weight  $\chi$ and consequently lowest weight is $w_{0}(\chi)$. Let $\{v_{j} :1\leq j \leq k\}$ be a basis of $H^{0}(G/P,\mathcal{L}(\chi))^{*}$  consisting of weights vectors  for the action of $T$, with $v_{1}$ is of weight $\chi$. Let $p_{id}$ denote the dual basis vector in $H^{0}(G/P,\mathcal{L}(\chi))$ corresponding to $v_{1}$.
		Since $(id)P/P\in X(w)$ and $p_{id}|_{BwP/P}\neq 0$, there is an element $x\in BwP/P$ such that $\hat{x}=\displaystyle\sum_{j=1}^{k}c_{j}v_{j}$ with $c_{1}\neq 0$.\\
		Let $\mu_{j}$ be the weight of $v_{j}$ for  $1\leq j\leq k$. Then, for all $1\leq j\leq k$, we have
		\[ (-\lambda_{s})(a)\cdot v_{j}=a^{-\langle \mu_{j},\lambda_{s}\rangle}\cdot v_{j} \ ,\ \ a\in \mathbb{C}^{\ast}\]
		Therefore, we have $\mu^{\mathcal{L}(\chi)}(x,-\lambda_{s})=-\displaystyle\min_{j}\{ -\langle \mu_{j},\lambda_{s}\rangle: c_{j}\neq 0\}$.\\
		Further, since $\chi$ is the highest weight of $H^{0}(G/P,\mathcal{L}(\chi))^{*}$, we have  $\langle\chi,-\lambda_{s}\rangle\leq \langle \mu_{j},-\lambda_{s}\rangle$ for all $1\leq j\leq k$. Since $c_{1}\neq 0$, we have $\mu^{\mathcal{L}(\chi)}(x,-\lambda_{s})=-\langle \chi,-\lambda_{s}\rangle=\langle\chi,\lambda_{s}\rangle$. 
		
		Since $\chi$ is  a positive integral linear combination of simple roots, we have $\langle \chi,\lambda_{s}\rangle \in \mathbb{N}$, and this implies \begin{equation}
			\mu^{\mathcal{L}(\chi)}(x,-\lambda_{s})>0
		\end{equation}
		Therefore,  from equation (3.3) and (3.4), we have $\mu^{\mathcal{L}(\chi)}(x,\pm\lambda_{s})>0$. Hence, the proof follows from $(1)$ of Corollary \ref{HM theorem}.
	\end{proof}	
	\begin{corollary}
		$X(w)^{ss}_{T}(\mathcal{L}(\chi))\neq\phi$ if and only if $X(w)^{ss}_{\lambda_{s}}(\mathcal{L}(\chi))\neq\phi$ for all $1\leq s\leq n$.
	\end{corollary}
	\begin{proof}
		By \cite[Lemma 2.1]{KP}, we have $X(w)^{ss}_{T}(\mathcal{L}(\chi))\neq \phi$ if and only if $w(\chi)\leq 0$. On the other hand, $w(\chi)\leq 0$ if and only if $\langle w(\chi),\lambda_{s}\rangle\leq0$ for all $1\leq s\leq n$. Now, the assertion of the corollary follows from Lemma \ref{main lemma}.
	\end{proof}
	\begin{lemma} Fix $1\leq s\leq n$. Let $w\in W^{J}$. Then $X(w)^{ss}_{\lambda_{s}}(\mathcal{L}(\chi))=X(w)^{s}_{\lambda_{s}}(\mathcal{L}(\chi))$ if and only if $\langle v(\chi),\lambda_{s}\rangle\neq 0$ for all $v\leq w$ in $W^{J}$.
		\label{ss=s}
	\end{lemma}
	\begin{proof}
		($:\Rightarrow$) Assume, on the contrary that there exists an element $v\leq w$ in $W^{J}$ such that  $\langle v(\chi), \lambda_{s}\rangle=0$.
		Take $x=vP/P$. Then $p_{v}\in H^{0}(X(w),\mathcal{L}(\chi))^{\lambda_{s}}$ and $p_{v}(x)\neq0$. Thus $x\in X(w)$ is a semistable point.
		From Lemma \ref{CSS}, we have $\mu^{\mathcal{L}}(x,\lambda_{s})=-\langle v(\chi),\lambda_{s}\rangle=0$.\\
		Therefore, $x$ is not a stable point (using Corollary \ref{HM theorem} (2)).\\
		($\Leftarrow$:) Conversely, assume that  $\langle v(\chi),\lambda_{s}\rangle\neq 0$ for all $v\leq w$ in $W^{J}$. Let $x\in X(w)^{ss}_{\lambda_{s}}(\mathcal{L}(\chi))$. We have $X(w)=\coprod_{v\leq w}BvP/P$.
		So, let $x\in Bv_{1}P/P$ for some $v_{1}\leq w$, $v_{1}\in W^{J}$.
		Using (1) of Corollary \ref{HM theorem}, we have  \begin{equation}
			\mu^{\mathcal{L}(\chi)}(x,\pm\lambda_{s})
			\geq 0
		\end{equation}
		\textbf{Claim 1:} $\mu^{\mathcal{L}(\chi)}(x,\lambda_{s})>0$\\
		\textbf{Proof of claim 1:}
		Using Lemma \ref{CSS}, we have  $\mu^{\mathcal{L}(\chi)}(x,\lambda_{s})=-\langle v_{1}(\chi),\lambda_{s}\rangle$. This is non zero, by hypothesis. Then inequality (3.5) implies  $\mu^{\mathcal{L}(\chi)}(x,\lambda_{s})>0$. 
		\\
		\textbf{Claim  2: }$\mu^{\mathcal{L}(\chi)}(x,-\lambda_{s})>0$\\
		\textbf{Proof of claim 2:}
		From  of \cite[Theorem 1.1(iii), p.500]{Deodhar}, we have $y\in B^{-}v_{2}P/P$ for some $v_{2}\in W^{J}$ such that $v_{2}\leq v_{1}$. Since $v_{2}\leq v_{1}\leq w$, by the given hypothesis 
		$\langle v_{2}(\chi),\lambda_{s}\rangle\neq 0$.
		Since $\lambda_{s}\in \overline{C(B)}$, using Lemma \ref{variation of Seshadri's lemma}, we have $\mu^{\mathcal{L}(\chi)}(x,-\lambda_{s})=\langle v_{2}(\chi),\lambda_{s}\rangle\neq 0$. 
		Now using inequality (3.5), we conclude that $\mu^{\mathcal{L}(\chi)}(x,-\lambda_{s})>0$. Thus, from Claim 1 and Claim 2, we have $\mu^{\mathcal{L}(\chi)}(x,\pm\lambda_{s})
		> 0$. Then Corollary \ref{HM theorem} (2) implies that $x$ is a stable point.	
	\end{proof}
	\begin{corollary} $(G/P)^{ss}_{\lambda_{s}}(\mathcal{L}(\chi))=(G/P)^{s}_{\lambda_{s}}(\mathcal{L}(\chi))$ if and only if $\langle v(\chi),\lambda_{s}\rangle\neq 0$ for all $v\in W^{J}$.
		\label{ss=s for G/P I}
	\end{corollary}
	\begin{proof}
		Let $w_{0}^{J}\in W^{J}$ denote the minimal coset representative of the longest element $w_{0}$. Then we have $X(w_{0}^{J})=G/P$. Further, we have $v\leq w_{0}^{J}$ for all $v\in W^{J}$ (see  \cite[p.44]{Coxeter group}). Then the corollary follows from Lemma \ref{ss=s}. 
	\end{proof}
	\begin{corollary}
		Let $G,P,\chi$ be as above. Then, we have $(G/P)^{ss}_{T}(\mathcal{L}(\chi))=(G/P)^{s}_{T}(\mathcal{L}(\chi))$ if and only if $(G/P)^{ss}_{\lambda_{s}}(\mathcal{L}(\chi))=(G/P)^{s}_{\lambda_{s}}(\mathcal{L}(\chi))$ for all $1\leq s\leq n$.
	\end{corollary}
	\begin{proof}
		By \cite[Lemma 4.1, p.38]{Torus quotient II}, we have $(G/P)^{ss}_{T}(\mathcal{L}(\chi))=(G/P)^{s}_{T}(\mathcal{L}(\chi))$ if and only if $\langle w(\chi),\lambda_{s}\rangle\neq 0$ for all $w\in W$ and for all $1\leq s\leq n$. Thus, the corollary follows from Corollary \ref{ss=s for G/P I}.
	\end{proof}
	\begin{corollary}$(G/P)^{ss}_{\lambda_{s}}(\mathcal{L}(\chi))=(G/P)^{s}_{\lambda_{s}}(\mathcal{L}(\chi))$ if and only if $\langle v(\chi),\lambda_{s}\rangle\neq 0$  for all minimal $v$ in $W^{J}$ such that $X(v)^{ss}_{\lambda_{s}}(\mathcal{L}(\chi))\neq \phi$.
		
		\label{ss=s for G/P II}
	\end{corollary}
	\begin{proof}
		($\Leftarrow:$)	By Corollary \ref{ss=s for G/P I}, it suffices to prove that $\langle v(\chi),\lambda_{s}\rangle\neq 0$  for all $v\in W^{J}$. Let $\tau\in W^{J}$ be an arbitrary element. If $X(\tau)^{ss}_{\lambda_{s}}(\mathcal{L}(\chi))= \phi$, then by Lemma \ref{main lemma}, we have $\langle \tau(\chi),\lambda_{s}\rangle>0$. Otherwise, $\tau\geq v$ for some minimal $v$ such that 
		$X(v)^{ss}_{\lambda_{s}}(\mathcal{L}(\chi))\neq \phi$. By hypothesis, $\langle v(\chi),\lambda_{s}\rangle\neq 0$. Hence, Lemma \ref{main lemma}  implies $\langle v(\chi),\lambda_{s}\rangle<0$. Further since $\tau\geq v$, we have $\tau(\chi)\leq v(\chi)$ (see \cite[Lemma 1.18 ]{LW}). Therefore, $\langle \tau(\chi),\lambda_{s}\rangle\leq \langle v(\chi),\lambda_{s}\rangle< 0$.
		\\
		(:$\Rightarrow$) Proof follows  from the above Corollary \ref{ss=s for G/P I}.	
	\end{proof}
	\begin{remark}
		For $G=SL(n,\mathbb{C})$, and $r,s$ be integers be such that $1\leq r,s\leq n-1$, we have $(G_{r,n})^{ss}_{\lambda_{s}}(\mathcal{L}(n\omega_{r}))=(G_{r,n})^{s}_{\lambda_{s}}(\mathcal{L}(n\omega_{r}))$ if and only if $n$ does not divide $rs$.
	\end{remark}
	\begin{proof}
		By \cite[Lemma 4.2, p.38]{Torus quotient II}, we have $\langle w(\omega_{r}),\lambda_{s}\rangle\neq 0$ for all $w\in W$ if and only if $n$ does not divide $rs$. Now, the proof follows from Corollary \ref{ss=s for G/P I}.
	\end{proof}
	\section{Minimal Schubert variety admitting semistable points in the minuscule case}
	\begin{lemma}
		Let $\omega_{r}$ be a minuscule fundamental weight $($i.e., $0\leq\langle w_{r},\beta\rangle\leq 1$ for all $\beta\in R^{+}$ $)$. Let $P=P_{S\setminus \{\alpha_{r}\}}$. Write $w_{0}^{S\setminus\{\alpha_{r}\}}(\omega_{r})=\omega_{r}-\displaystyle\sum_{j=1}^{n}a_{j}\alpha_{j}$, $a_{j}\in \mathbb{Z}_{\geq0}$ for all $1\leq j\leq n$.
		Fix $1\leq s \leq n$ and $c$ an integer such that $0\leq c \leq a_{s}$.	Then there is a unique minimal element $\tau_{s,c}\in W^{S\setminus\{\alpha_{r}\}}$ such that $\langle \omega_{r}-\tau_{s,c}(\omega_{r}), \lambda_{s}\rangle=c$.
		\label{main lemma2}
	\end{lemma}
	\begin{proof}
		From \cite[Proposition 2.1, p.725]{Stem}, any two reduced expressions of $w_{0}^{S\setminus\{\alpha_{r}\}}$ differ only by commuting relations. Therefore, for every integer  $1\leq j\leq n$, the number of times $s_{j}$ appears in a reduced expression of $w_{0}^{S\setminus\{\alpha_{r}\}}$ is exactly $a_{j}$.\\
		Let $k=l(w_{0}^{S\setminus\{\alpha_{r}\}})$. So, let $1\leq l_{1}\leq k$ be the largest integer such that there is a reduced expression $w_{0}^{S\setminus\{\alpha_{r}\}}=s_{i_{1}}s_{i_{2}}\cdots s_{i_{l_{1}-1}}s_{i_{l_{1}}}\cdots s_{i_{k}}$ such that \[\langle \omega_{r}-s_{i_{l_{1}}}\cdots s_{i_{k}}(\omega_{r}), \lambda_{s}\rangle=1\]
		Let $1\leq l_{2}< l_{1}$ be the largest integer such that there is a reduced expression $w_{0}^{S\setminus\{\alpha_{r}\}}=s_{i_{1}}s_{i_{2}}\cdots s_{i_{l_{2}-1}}s_{i_{l_{2}}}\cdots s_{i_{l_{1}-1}}s_{i_{l_{1}}}\cdots s_{i_{k}}$ such that \[\langle \omega_{r}-s_{i_{l_{2}}}\cdots s_{i_{l_{1}}}\cdots s_{i_{k}}(\omega_{r}), \lambda_{s}\rangle=2\]
		Proceeding like this, we can obtain a decreasing sequence $k\geq l_{1}>l_{2}>l_{3}>\cdots >l_{a_{s}}\geq 1$ of positive integers such that for any integer $1\leq c\leq a_{s}$, \[\tau_{s,c}=s_{i_{l_{c}}}\cdots s_{i_{l_{c-1}}}\cdots s_{i_{l_{2}}}\cdots s_{i_{l_{1}-1}}s_{i_{l_{1}}}\cdots s_{i_{k}}\] is the unique minimal element of $W^{S\setminus\{\alpha_{r}\}}$ such that 
		\[\langle \omega_{r}-\tau_{s,c}(\omega_{r}), \lambda_{s}\rangle=c.\]
	\end{proof}
	\begin{remark} For $0\leq c_{1}\leq c_{2}\leq a_{s}$, we have $\tau_{s,c_{1}}\leq \tau_{s,c_{2}}$.
		\label{remark in min element}
	\end{remark}
	\begin{corollary} 
		Let $\omega_{r}$ be a minuscule fundamental weight. Let  $P=P_{S\setminus\{\alpha_{r}\}}$. 	Let $m\in \mathbb{N}$ be the least positive integer such that $m\omega_{r}$ is in the root lattice. Fix $1\leq s\leq n$. Then there is a unique minimal element $w_{s,r}\in W^{S\setminus\{\alpha_{r}\}}$ such that $X(w_{s,r})^{ss}_{\lambda_{s}}(\mathcal{L}(m\omega_{r}))\neq\phi$.
		\label{unique min}
	\end{corollary}
	\begin{proof}
		By Lemma \ref{main lemma}, it suffices to prove that there is a unique minimal element $w_{s,r}\in W^{S\setminus\{\alpha_{r}\}}$  such that $\langle w_{s,r}(m\omega_{r}), \lambda_{s}\rangle\leq 0$. \\
		Write $m\omega_{r}=\displaystyle\sum_{j=1}^{n}m_{j}\alpha_{j}$, where $m_{j}\in \mathbb{N}$ for all $1\leq j\leq n$.\\
		With the notation as in the Lemma \ref{main lemma2}, \[w_{0}^{S\setminus\{\alpha_{r}\}}(m\omega_{r})=m\omega_{r}-\displaystyle\sum_{j=1}^{n}ma_{j}\alpha_{j}=\displaystyle\sum_{j=1}^{n}(m_{j}-ma_{j})\alpha_{j}\] Since $w_{0}^{S\setminus\{\alpha_{r}\}}(m\omega_{r})=w_{0}(m\omega_{r})\leq 0$, we have $a_{s}m\geq m_{s}$. 	Let $q\in\mathbb{N}$ be the least positive integer such that $qm\geq m_{s}$. Then, we have $1\leq q\leq a_{s}$. By Lemma \ref{main lemma2} above, there is a unique minimal element $\tau_{s,q}\in W^{S\setminus\{\alpha_{r}\}}$  such that \begin{equation}
			\langle\omega_{r}-\tau_{s,q}(\omega_{r}), \lambda_{s}\rangle=q
		\end{equation}
		Therefore, $\langle \tau_{s,q}(m\omega_{r}), \lambda_{s}\rangle=\langle m \omega_{r}, \lambda_{s}\rangle-qm$= $m_{s}-qm\leq 0$. 
		\\
		Take $w_{s,r}=\tau_{s,q}$.
		Now we prove the uniqueness of $w_{s,r}$.
		Let $v\in W^{{S\setminus\{\alpha_{r}\}}}$ be a minimal element such that $\langle v(m\omega_{r}), \lambda_{s}\rangle\leq 0$. Let $v(\omega_{r})=\omega_{r}-\displaystyle\sum_{j=1}^{n}b_{j}\alpha_{j}$, $b_{j}\in \mathbb{Z}_{\geq0}$, for $1\leq j\leq n$.
		We have $\langle v(m\omega_{r}), \lambda_{s}\rangle=m_{s}-mb_{s}\leq 0$. This implies $b_{s}m\geq m_{s}$. Since $q$ is the least positive integer such that $qm\geq m_{s}$, we get \begin{equation}
			b_{s}\geq q
		\end{equation}
		
		Also $v\leq w_{0}^{S\setminus \{\alpha_{r}\}}$. Since $m\omega_{r}$ is dominant we have $ v(m\omega_{r})\geq w_{0}^{S\setminus \{\alpha_{r}\}}(m\omega_{r})$ (see \cite[Lemma 1.18, p.183]{LW}). This implies $m_{s}-mb_{s}\geq m_{s}-ma_{s}$. Hence \begin{equation}
			b_{s}\leq a_{s}
		\end{equation}	
		Therefore, from (4.2) and (4.3), we have  $1\leq q\leq b_{s}\leq a_{s}$. Note that $v$ is a minimal element such that 
		$\langle \omega_{r}-v(\omega_{r}),\lambda_{s}\rangle=b_{s}$. Therefore, by uniqueness of $\tau_{s,b_{s}}$ in Lemma \ref{main lemma2}, we have  $\tau_{s,b_{s}}=v$. Then  Remark \ref{remark in min element} above implies that $w_{s,r}\leq v$. Therefore, using minimality of $v$, we have $w_{s,r}=v$. Hence, $X(w_{s,r})$ is the unique minimal dimensional Schubert variety admitting semistable point.
	\end{proof}
	\section{Description  of $X(w_{s,r})$}
	In this section, we will give the description of the $w_{s,r}$ for all minuscule fundamental weights $\omega_{r}$ and for all $1\leq s\leq n$. Let $m$ be a positive integer such that $m\omega_{r}$ is in the root lattice.
	\begin{lemma}
		Let $w\in W^{S\setminus\{\alpha_{r}\}}$ be such that $w^{-1}(\alpha_{q})<0$, for some $\alpha_{q}\in S$. Then we have $s_{q}w\in W^{S\setminus\{\alpha_{r}\}}$ and $l(s_{q}w)=l(w)-1$. 
		\label{schubert divisor}
		\begin{proof}
			Note that $w(\alpha)>0$ for all $\alpha\in S\setminus \{\alpha_{r}\}$. Since $w^{-1}(\alpha_{q})<0$, $w(\alpha)\neq \alpha_{q}$ for all $\alpha\in S\setminus\{\alpha_{r}\}$. Therefore, $s_{q}w(\alpha)>0$ for all $\alpha\in S\setminus\{\alpha_{r}\}$ (using   \cite[Lemma B, p. 50]{Hum1}) and hence, $s_{q}w\in W^{S\setminus\{\alpha_{r}\}}$. From \cite[Lemma 1.6, p.12]{Hum3}, we have $l(s_{q}w)=l(w)-1$.
		\end{proof}
	\end{lemma}
	\begin{lemma}
		For any reduced expression $w_{s,r}=s_{i_{1}}\cdots s_{i_{k}}$, we have $i_{1}=s$.
		\label{reduced expression}
	\end{lemma}
	\begin{proof}
		Assume, on the contrary, that there exists a reduced expression $w_{s,r}=s_{j_{1}}\cdots s_{j_{k}}$ such that $j_{1}\neq s$. 
		From \cite[Corollary 10.2]{Hum1}, we have $w_{s,r}^{-1}(\alpha_{j_{1}})<0$. Then, by Lemma \ref{schubert divisor}, $X(s_{j_{1}}w_{s,r})$ is a Schubert divisor in $X(w_{s,r})$.
		Now $s_{j_{1}}(\lambda_{s})=\lambda_{s}-\langle\alpha_{j_{1}},\lambda_{s}\rangle\alpha_{j_{1}}=\lambda_{s}-\delta_{j_{1},s}\alpha_{j_{1}}=\lambda_{s}$. Since $\langle.,.\rangle$ is $W$-invariant, we have 
		$\langle s_{j_{1}}w_{s,r}(\omega_{r}), \lambda_{s}\rangle=\langle w_{s,r}(\omega_{r}), \lambda_{s}\rangle$. Then Lemma \ref{main lemma} implies $\langle s_{j_{1}}w_{s,r}(\omega_{r}), \lambda_{s}\rangle\leq 0$.\\
		Again using Lemma \ref{main lemma}, we have $X(s_{j_{1}}w_{s,r})^{ss}_{\lambda_{s}}(\mathcal{L}(m\omega_{r}))\neq \phi$. This gives a contradiction, since  dim $X(s_{j_{1}}w_{s,r})<$ dim $X(w_{s,r})$.
	\end{proof}
	\begin{corollary}
		$w_{s,r}^{{-1}}\in W^{S\setminus\{
			\alpha_{s}\}}$ for all $1\leq s,r\leq n$.
		\label{unique pick}
	\end{corollary}
	\begin{proof}
		Let $w_{s,r}^{{-1}}=s_{i_{k}}\cdots s_{i_{1}}$ be a reduced expression. Then Lemma \ref{reduced expression} implies $i_{1}=s$. From \cite[Corollary 10.2]{Hum1}, $w_{s,r}^{-1}(\alpha_{s})<0$. We prove that $w_{s,r}^{-1}(\alpha_{q})>0$ for all $\alpha_{q}\in S\setminus\{\alpha_{s}\}$.
		Assume that  $w_{s,r}^{-1}(\alpha_{q})<0$ for some $q\neq s$. Then from Lemma \ref{schubert divisor}, $s_{q}w_{s,r}\in W^{S\setminus \{\alpha_{r}\}}$ and dim $X(s_{q}w_{s,r})$=dim $X(w_{s,r})-1$. Since $q\neq s$, we have  $\langle s_{q}w_{s,r}(\omega_{r}), \lambda_{s}\rangle=\langle w_{s,r}(\omega_{r}), \lambda_{s}\rangle
		\leq 0$. Hence $X(s_{q}w_{s,r})^{ss}_{\lambda_{s}}(\mathcal{L}(m\omega_{r}))\neq \phi$ (see Lemma \ref{main lemma}). This gives a contradiction, because $X(w_{s,r})$ is the minimal dimensional Schubert variety admitting semistable point. Hence, $w_{s,r}^{-1}(\alpha_{q})>0$ for all $\alpha_{q}\in S\setminus\{\alpha_{s}\}$. Thus $w_{s,r}^{-1}\in W^{S\setminus\{\alpha_{s}\}}$.
	\end{proof}
	\begin{lemma} 
		Fix $1\leq s,r\leq n$. Then we have	\begin{enumerate}[label=(\roman*)]
			\item	$-1< \langle w_{s,r}(\omega_{r}),\lambda_{s}\rangle\leq0$. \item $\langle w_{s,r}(\omega_{r}),\lambda_{s}\rangle\neq0$ if and only if $X(w_{s,r})^{ss}_{\lambda_{s}}(\mathcal{L}(m\omega_{r}))=X(w_{s,r})^{s}_{\lambda_{s}}(\mathcal{L}(m\omega_{r}))$.
		\end{enumerate}
		\label{ss=s,min}
	\end{lemma}
	\begin{proof} Proof of $(i):$ Let $w_{s,r}=s_{i_{1}}\cdots s_{i_{k}}$ be a reduced expression. Then in view of Lemma \ref{reduced expression}, we have $i_{1}=s$.  Moreover,
		$X(s_{s}w_{s,r})$ is a Schubert divisor in $X(w_{s,r})$. Then, we have $\langle w_{s,r}(\omega_{r}),\alpha_{s}\rangle=-1=-\langle s_{s}w_{s,r}(\omega_{r}),\alpha_{s}\rangle$. \\
		Now $s_{s}w_{s,r}(\omega_{r})=w_{s,r}(\omega_{r})-\langle w_{s,r}(\omega_{r}),\alpha_{s}\rangle\alpha_{s}$. Since $X(s_{s}w_{s,r})^{ss}_{\lambda_{s}}(\mathcal{L}(m\omega_{r}))=\phi$, by Lemma \ref{main lemma} we have,  $\langle s_{s}w_{s,r}(\omega_{r}),\lambda_{s}\rangle>0$. Therefore, \[\langle w_{s,r}(\omega_{r}),\lambda_{s}\rangle=\langle s_{s}w_{s,r}(\omega_{r}),\lambda_{s}\rangle-\langle s_{s}w_{s,r}(\omega_{r}),\alpha_{s}\rangle=\langle s_{s}w_{s,r}(\omega_{r}),\lambda_{s}\rangle-1 > -1\]
		Proof of $(ii):$ Using Lemma \ref{ss=s}, it suffices to prove that $\langle v(\omega_{r}),\lambda_{s}\rangle \neq 0$ for all $v\leq w_{s,r}$. Since $X(w_{s,r})$ is a minimal dimensional Schubert variety admitting semistable point, Lemma \ref{main lemma} implies $\langle v(\omega_{r}),\lambda_{s}\rangle>0$ for all $v<w_{s,r}$. Therefore, $\langle v(\omega_{r}),\lambda_{s}\rangle \neq 0$ for all $v\leq w_{s,r}$ if and only if $\langle w_{s,r}(\omega_{r}),\lambda_{s}\rangle \neq 0$.	\end{proof}
	In minuscule case, $X(w)^{ss}_{\lambda_{s}}(\mathcal{L}(m\omega_{r}))=X(w)^{s}_{\lambda_{s}}(\mathcal{L}(m\omega_{r}))$ for the unique minimal Schubert variety for which $X(w)^{ss}_{\lambda_{s}}(\mathcal{L}(m\omega_{r}))\neq \phi$ if and only if $(G/P)^{ss}_{\lambda_{s}}(\mathcal{L}(m\omega_{r}))=(G/P)^{s}_{\lambda_{s}}(\mathcal{L}(m\omega_{r}))$. In general minimal dimensional Schubert variety admitting semistable point for the action of $\lambda_{s}$ with respect to a linearized line bundle $\mathcal{L}(\chi)$ need not be unique. So, there is a natural question; is it possible that there may be Schubert variety $X(w)\subseteq G/P$ and a linearized line bundle $\mathcal{L}(\chi)$  such that $X(w)^{ss}_{\lambda_{s}}(\mathcal{L}(\chi))=X(w)^{s}_{\lambda_{s}}(\mathcal{L}(\chi))\neq \phi$, but  $(G/P)^{ss}_{\lambda_{s}}(\mathcal{L}(\chi))\neq(G/P)^{s}_{\lambda_{s}}(\mathcal{L}(\chi))$ ? \\
	The following example illustrates this situation.
	\begin{example}
		Let $G=SL(5,\mathbb{C})$ and $P_{\alpha_{1}}$ be the minimal parabolic subgroup corresponding to root $\alpha_{1}$. Let $\chi=2\omega_{2}+2\omega_{3}+5\omega_{4}$. Note that $\chi=3\alpha_{1}+6\alpha_{2}+7\alpha_{3}+6\alpha_{4}$. Let $w_{1}=s_{2}s_{1}s_{3}s_{2}$, $w_{2}=s_{2}s_{3}s_{4}$. Then $w_{1}(\chi)=\chi-6\alpha_{2}-4\alpha_{3}-2\alpha_{1}$ and $w_{2}(\chi)=\chi-5\alpha_{4}-7\alpha_{3}-9\alpha_{2}$. Therefore,
		\begin{equation}
			\langle w_{1}(\chi),\lambda_{2}\rangle=0
		\end{equation}
		\begin{equation}
			\langle w_{2}(\chi),\lambda_{2}\rangle=-3
		\end{equation}
		Using equation (5.1) and (5.2), both $X(w_{1})$ and $X(w_{2})$ are minimal dimensional Schubert varieties such that $X(w_{1})^{ss}_{\lambda_{2}}(\mathcal{L}(\chi))\neq \phi$, $X(w_{2})^{ss}_{\lambda_{2}}(\mathcal{L}(\chi))\neq \phi$ (see Lemma \ref{main lemma}). Since $\langle v(\chi),\lambda_{2}\rangle\neq 0$ for all $v\leq w_{2}$, by Lemma \ref{ss=s}, we have $X(w_{2})^{ss}_{\lambda_{2}}(\mathcal{L}(\chi))=X(w_{2})^{s}_{\lambda_{2}}(\mathcal{L}(\chi))$. But equation (5.1) implies $(G/P_{\alpha_{1}})^{ss}_{\lambda_{2}}(\mathcal{L}(\chi))\neq (G/P_{\alpha_{1}})^{s}_{\lambda_{2}}(\mathcal{L}(\chi))$ (see Corollary \ref{ss=s for G/P II}).
	\end{example}
	\subsection{Type $A_{n-1}$}
	Note that if $G$ is of type $A_{n-1}$, then every fundamental weight is minuscule (see \cite[p.180]{LW}). Fix an integer $1\leq r\leq n-1$.
	\begin{lemma}
		Let $w\in  W^{S\setminus\{\alpha_{r}\}}$,  $w\neq id$. Then there exists an $i\in \mathbb{N}, i\leq r$ and a sequence $\{a_{j}: i\leq j\leq r\}$ of positive integers  such that the following holds.
		\begin{enumerate}
			\item	$a_{j}\geq j$, for  $i\leq j\leq r$.
			\item $a_{j}<a_{j+1}$ for $i\leq j\leq r-1$.
			\item$w=(s_{a_{i}}s_{a_{i}-1}...s_{i})(s_{a_{i+1}}s_{a_{i+1}-1}...s_{i+1})...(s_{a_{r}}s_{a_{r}-1}...s_{r})$ with
			$l(w)=\displaystyle\sum_{j=i}^{r}(a_{j}-j+1)$.
		\end{enumerate}
		\label{type An min coset}
	\end{lemma}
	\begin{proof}	
		For proof we refer \cite[Lemma 2.1, p. 85]{s kannan}.
	\end{proof}	
	In view of Lemma \ref{type An min coset}, there is a reduced expression\\ $(s_{a_{i}}...s_{i})(s_{a_{i+1}}...s_{i+1})...(s_{a_{r-1}}...s_{r}s_{r-1})(s_{a_{r}}...s_{r+1}s_{r})$  of $w_{s,r}$ with $i\leq a_{i}< a_{i+1}<\cdots<a_{r}\leq n-1$.
	\begin{lemma}   We have $a_{i}=s$ and  $a_{j+1}-a_{j}=1$ for all $j=i,i+1,\ldots,r-1$.
		\label{type An}
	\end{lemma}
	\begin{proof}
		First note that $a_{i}=s$ by Lemma \ref{reduced expression}. If $a_{k+1}-a_{k}\geq 2$ for some $1\leq k\leq r-1$, let $k_{0}$ be the least such integer. Then $w_{s,r}=s_{a_{k_{0}+1}}\cdot v$ for some $v\in W^{S\setminus\{\alpha_{r}\}}$ such that $l(w_{s,r})=l(v)+1$. From  \cite[Corollary 10.2]{Hum1}, we have $w_{s,r}^{-1}(\alpha_{k_{0}+1})<0$. By property $(2)$ of Lemma \ref{type An min coset},  $a_{k_{0}+1}>a_{i}$ and hence $a_{k_{0}+1}\neq s$. This leads to a contradiction to Corollary \ref{unique pick}.
	\end{proof}
	\begin{lemma} Let $1\leq s,r\leq n-1$. Take $p=\lfloor \frac{rs}{n}\rfloor$. Then we have,\\ $w_{s,r}=(s_{s}...s_{p+1})(s_{s+1}...s_{p+2})...(s_{s+r-p-1}...s_{r})$.\label{criterion An}
	\end{lemma}	
	\begin{proof}By Lemma \ref{type An}, we have $w_{s,r}=(s_{s}...s_{i+1}s_{i})(s_{s+1}...s_{i+1})..(s_{s+j-i}...s_{j})..(s_{s+r-i}...s_{r})$. It suffices to prove that $i=p+1$. For $1\leq k\leq n$, let $n_{k}$ denotes the number of times simple reflection $s_{k}$ appears in a reduced expression of $w_{s,r}$. Then $w_{s,r}(\omega_{r})=\omega_{r}-\displaystyle\sum_{k=1}^{n}n_{k}\alpha_{k}$. We have,  $\omega_{r}=\frac{1}{n}(\displaystyle\sum_{i=1}^{r-1}i(n-r)\alpha_{i}+\displaystyle\sum_{i=r}^{n-1}r(n-i)\alpha_{i})$ (see \cite[p. 69]{Hum1}).\\
		\textbf{Case 1:} $s\geq r$.  In this case 
		$\langle w_{s,r}(\omega_{r}),\lambda_{s}\rangle=\frac{r(n-s)}{n}-n_{s}$. For every $i\leq j\leq r$, we have  $s\ge r\geq j$. Thus, we have $s+j-i\geq s\geq j$, for all $i\leq j\leq r$. Therefore, $n_{s}=r-i+1$. Hence, $\langle w_{s,r}(\omega_{r}),\lambda_{s}\rangle=-\frac{rs}{n}+i-1$. Therefore, from
		Lemma \ref{ss=s,min} $(i)$, we have \begin{equation}
			-1<-\frac{rs}{n}+i-1\leq 0
		\end{equation}
		\textbf{Case 2:} $s<r$. In this case $\langle w_{s,r}(\omega_{r}),\lambda_{s}\rangle=\frac{s(n-r)}{n}-n_{s}$.
		Since $j\leq s\leq s+j-i$ only for $i\leq j\leq s$, we have $n_{s} =s-i+1$.  Therefore, $\langle w_{s,r}(\omega_{r}),\lambda_{s}\rangle=\frac{s(n-r)}{n}-n_{s}=-\frac{rs}{n}+i-1$.\\
		Now using $(i)$ of Lemma \ref{ss=s,min}, we have
		\begin{equation}		
			-1<-\frac{rs}{n}+i-1\leq0
		\end{equation}
		From expression (5.3) and (5.4), we  see that in both the cases  $\frac{rs}{n}<i\leq \frac{rs}{n}+1$. This implies $i=\lfloor \frac{rs}{n}\rfloor+1$. Moreover, for all $1\leq s,r\leq n-1$ we  get \begin{equation}
			\langle w_{s,r}(\omega_{r}),\lambda_{s}\rangle=-\frac{rs}{n}+\lfloor\frac{rs}{n}\rfloor.
		\end{equation}
	\end{proof}
	\begin{corollary} Fix $1\leq s\leq n-1$. Then
		$X(w_{s,r})^{ss}_{\lambda_{s}}(\mathcal{L}(n\omega_{r}))=X(w_{s,r})^{s}_{\lambda_{s}}(\mathcal{L}(n\omega_{r}))$ if and only if $n\nmid rs$.
		\label{tpye An:ss=s}
	\end{corollary}
	\begin{proof}
		From expression (5.5), we see that $n\nmid rs$ if and only if $\langle w_{s,r}(\omega_{r}),\lambda_{s}\rangle\neq 0$.  
		Now from Lemma \ref{ss=s,min} $(ii)$, we have  $\langle w_{s,r}(\omega_{r}),\lambda_{s}\rangle\neq 0$ if and only if $X(w_{s,r})^{ss}_{\lambda_{s}}(\mathcal{L}(n\omega_{r}))=X(w_{s,r})^{s}_{\lambda_{s}}(\mathcal{L}(n\omega_{r}))$.
	\end{proof}
	\subsection{Type $B_{n}$}
	Note that  $\omega_{n}$ is the only minuscule fundamental weight in type $B_{n}$ (see \cite[p. 180]{LW}).\\
	Let $w_{j}=s_{n}s_{n-1}\cdots s_{j+1}s_{j}$ for $1\leq j\leq n$.
	Define  $w_{j}(0)=id$, $w_{j}(1)=s_{j}$ and $w_{j}(l_{j})=s_{j+l_{j}-1}\cdots s_{j+1}s_{j}$ for $2\leq l_{j}\leq n+1-j$.\\
	We recall the following theorem from \cite{min coset} (see \cite[Theorem 6, p.708]{min coset}). Note that the our convention is left coset, whereas in \cite{min coset} it was proved for right coset. So, the theorem is stated with suitable modification.
	\begin{theorem}
		$W^{S\setminus\{\alpha_{n}\}}=\{w_{1}(l_{1})w_{2}(l_{2})\cdots w_{n-1}(l_{n-1})w_{n}(l_{n}):$
		
		\ $(1)$ $0\leq l_{k}\leq n+1-k$,
		$(2)$ $l_{k-1}\leq l_{k}+1$ \ and \
		$(3)$ $l_{k-1}\leq l_{k}$, \ if   $l_{k}\leq n-k\}$
		\label{type Bn}
	\end{theorem}
	\begin{lemma}
		Write $s=n+1-j$, for $j=1,\ldots,,n$. Let $p=\lceil \frac{s}{2}\rceil$. Then
		\begin{enumerate}
			\item If $j<p$, then\\
			$w_{s,n}=w_{1}(0)\ldots w_{n-(j+p-1)}(0)w_{n-(j+p-2)}(p)w_{n-(j+p-3)}(p)...w_{n-p}(p)w_{n-(p-1)}(p)\\w_{n-(p-2)}(p-1)\ldots w_{n-j}(j+1)w_{n-(j-1)}(j)w_{n-(j-2)}(j-1)...w_{n-1}(2)w_{n}(1)$
			\item If $j\geq p$ , then\\
			$w_{s,n}=w_{1}(0)\ldots w_{n-(j+p-1)}(0)w_{n-(j+p-2)}(p)...w_{n-j}(p)w_{n-(j-1)}(p)w_{n-(j-2)}(p)\\w_{n-(j-3)}(p)...w_{n-(p-1)}(p)w_{n-(p-2)}(p-1)w_{n-(p-3)}(p-2)...w_{n-1}(2)w_{n}(1)$
			\end{enumerate}
		\label{Type Bn criteion}
	\end{lemma}
	\begin{proof}  Recall from \cite[p. 69]{Hum1} that  
		$\omega_{n}=\frac{1}{2}(\displaystyle\sum_{i=1}^{n}i\alpha_{i})$.\\
		Let $w_{s,n}=w_{1}(l_{1})w_{2}(l_{2})....w_{n-1}(l_{n-1})w_{n}(l_{n})$ be satisfying conditions of the Theorem \ref{type Bn}. 
		Proof of $(1):$ If $j<p$ then, $l_{n-(j-1)}$ cannot be less than  $j-1$.
		For otherwise, if $l_{n-(j-1)}\leq j-1$, then condition (3) implies $l_{n-(j+p-2)}\leq j-1<p$. In that case $w_{n-(j+p-2)}(l_{n-(j+p-2)})$ will not include the simple reflection $s_{n-(j-1)}$.
		Similar reasoning shows that $l_{n-(j-2)}=j-1$ and so on up to $l_{1}=1$.
		
		$s_{n-(j-1)}$ must appear $p$ times in any reduced expression of $w_{s,n}$. Indeed, $s_{n-(j-1)}$ appears in $w_{i}$ for all $n-(j+p-2)\leq i\leq n-(j-1)$.\\
		Proof of $(2):$
		If $l_{n-(p-1)}$ is less than $p$, then using condition $(3)$ of Theorem \ref{type Bn}, we can see that $l_{n-(j+p-2)}<p$. Hence, $w_{n-(j+p-2)}(l_{n-(j+p-2)})$ will not include the simple reflection $s_{n-(j-1)}$. This implies $l_{n-(p-1)}=p$ and similarly $l_{n-i}=i+1$ for all $0\leq i\leq p-2$. 
		By condition $(1)$ of the Theorem $l_{n-(j-1)}\leq j$. Since $p-1\leq j-1$,
		if $l_{n-(j-1)}\leq p-1$, then condition $(3)$ implies  $l_{n-(j+p-2)}<p$. This is not possible, because $w_{n-(j+p-2)}$ must conatin $s_{n-(j-1)}$. Therefore, $l_{n-(j-1)}\geq p$. \\
		Then minimality of $w_{s,n}$ implies $l_{n-(j-1)} =p$. Note that $s_{n-(j-1)}$ appears in $w_{i}$ for all $n-(j+p-2)\leq i\leq n-(j-1)$ and there are $p$ many of them. This completes the proof.
	\end{proof}
	\begin{lemma}
		$X(w_{s,n})^{ss}_{\lambda_{s}}(\mathcal{L}(2\omega_{n}))=X(w_{s,n})^{s}_{\lambda_{s}}(\mathcal{L}(2\omega_{n}))$ if and only if $s \not\equiv 0 \ (mod\ 2)$.
		\label{ss=s,type Bn}
	\end{lemma} 
	\begin{proof}	Fix $1\leq s\leq n$. Let $w_{s,n}(\omega_{n})=\omega_{n}-\displaystyle\sum_{k=1}^{n}a_{k}\alpha_{k}$, $a_{k}\in \mathbb{Z}_{\geq 0}$. Let $p=\lceil \frac{s}{2}\rceil$. Then, we have 
		\begin{equation}
			a_{k}=
			\begin{cases}
				0 & \text{if} \ 1\leq k\leq s-p\\
				k-s+p & \text{if } \ s-(p-1)\leq k\leq s-1\\
				p & \text{if} \ s\leq k\leq n
			\end{cases}
		\end{equation}
		Therefore, $\langle w_{s,r}(\omega_{r}),\lambda_{s}\rangle=\frac{s}{2}-\lceil \frac{s}{2}\rceil$. This is non-zero if and only if $s \not\equiv 0\ (mod\ 2)$. Then the proof follows from Lemma \ref{ss=s,min} $(ii)$.
	\end{proof}
	\subsection{Type $C_{n}$}
	Note that $\omega_{1}$ is the only minuscule fundamental weight in type $C_{n}$ (see \cite[p.180]{LW}).
	\begin{lemma}  For $1\leq s\leq n$,  $w_{s,1}=s_{s}s_{s-1}\cdots s_{3}s_{2}s_{1}$.
	\end{lemma}
	\begin{proof}
		Recall from \cite[p.69]{Hum1}, $\omega_{1}=\alpha_{1}+\alpha_{2}+\cdots+ \alpha_{n-1}+\frac{1}{2}\alpha_{n}$. Note that $w_{0}^{S\setminus\{\alpha_{1}\}}= s_{1}s_{2}\cdots s_{n-1}s_{n}s_{n-1}s_{n-2}\cdots s_{2}s_{1}$ (see \cite[Theorem 6, p.708]{min coset}). Then the proof follows easily.
	\end{proof}
	\begin{lemma}
		$X(w_{s,1})^{ss}_{\lambda_{s}}(\mathcal{L}(2\omega_{1}))=X(w_{s,1})^{s}_{\lambda_{s}}(\mathcal{L}(2\omega_{1}))$ if and only if $s=n$.
	\end{lemma}
	\begin{proof}
		Note that $w_{s,1}(\omega_{1})=\alpha_{s+1}+\cdots+\alpha_{n-1}+\frac{1}{2}\alpha_{n}$ for $1\leq s\leq n-2$, $w_{n-1,1}(\omega_{1})=\frac{1}{2}\alpha_{n}$ and $w_{n,1}(\omega_{1})=-\frac{1}{2}\alpha_{n}$. Therefore, 
		$\langle w_{s,1}(\omega_{1}),\lambda_{s}\rangle=0$ for $1\leq s\leq n-1$ and $\langle w_{s,1}(\omega_{1}),\lambda_{s}\rangle=-\frac{1}{2}$. Then the proof follows from Lemma \ref{ss=s,min} $(ii)$.
	\end{proof}
	\subsection{Type $D_{n}$}
	Note that $\omega_{1}, \omega_{n-1}, \omega_{n}$ are the only minuscule fundamental weights (see \cite[p.180]{LW}).
	\subsubsection{Description of $w_{s,1}$}
	\begin{lemma} \begin{enumerate}
			\item  $s\leq n-2$, $w_{s,1}=s_{s}s_{s-1}\cdots s_{3}s_{2}s_{1}$
			\item $s=n-1$, $w_{s,1}=s_{n-1}s_{n-2}s_{n-3}\cdots s_{2}s_{1}$
			\item $s=n$, $w_{s,1}=s_{n}s_{n-2}s_{n-3}\cdots s_{2}s_{1}$
		\end{enumerate}
	\end{lemma}
	\begin{proof}
		$\omega_{1}=\alpha_{1}+\cdots+\alpha_{n-2}+\frac{1}{2}(\alpha_{n-1}+\alpha_{n-2})$ (see \cite[p.69]{Hum1}). From \cite[Theorem 4]{min coset},  we have $w_{0}^{S\setminus\{\alpha_{1}\}}= s_{1}s_{2}\cdots s_{n-2}s_{n-1}s_{n}s_{n-2}s_{n-3}\cdots s_{3}s_{2}s_{1}$. Further, note that $s_{n-1}$ and $s_{n}$ commute. Then the proof follows easily. 
	\end{proof}
	\begin{lemma}
		$X(w_{s,1})^{ss}_{\lambda_{s}}(\mathcal{L}(2\omega_{1}))=X(w_{s,1})^{s}_{\lambda_{s}}(\mathcal{L}(2\omega_{1}))$ if and only if $s=n-1,n.$
	\end{lemma}
	\begin{proof}
		Note that \[w_{s,1}(\omega_{1})=\alpha_{s+1}+\cdots+\alpha_{n-2}+\frac{1}{2}(\alpha_{n-1}+\alpha_{n}) \ for\  1\leq s\leq n-3\]
		\[w_{n-2,1}(\omega_{1})=\frac{1}{2}(\alpha_{n-1}+\alpha_{n})\]	\[w_{n-1,1}(\omega_{1})=-\frac{1}{2}\alpha_{n-1}+\frac{1}{2}\alpha_{n}\]
		\[w_{n,1}(\omega_{1})=-\frac{1}{2}(\alpha_{n-1}+\alpha_{n})\]
		Therefore,	$\langle w_{s,1}(\omega_{1}),\lambda_{s}\rangle=0$ for $1\leq s\leq n-2$ and $\langle w_{s,1}(\omega_{1}),\lambda_{s}\rangle=-\frac{1}{2}$ for $s=n-1,n$. Now the proof follows from Lemma \ref{ss=s,min}$(ii)$.
	\end{proof}
	\subsubsection{Description of $w_{s,n}$} Let 	$w_{n}=s_{1}\cdots s_{n-3}s_{n-2}s_{n} \ and \ 
	w_{n-1}=s_{1}\cdots s_{n-3}s_{n-2}s_{n-1}$.
	Following \cite{min coset} (see \cite[Theorem 4, p.703]{min coset}), we define
	\begin{align*}
		w_{n}(l)=s_{l-1}\cdots s_{n-3}s_{n-2}s_{n}\\
		w_{n-1}(l)=s_{l-1}\cdots s_{n-3}s_{n-2}s_{n-1}
	\end{align*}
	for $l=2,3,\ldots,n-1 $
	and $w_{n}(n)=s_{n}$, $w_{n-1}(n)=s_{n-1}$. Note that we are folllowing the convention of \cite{Hum1} for numbering of Dynkin diagram, so the following result is stated from \cite{min coset} with appropiate modification (see \cite[comment after Theorem 4, p.708]{min coset}).
	\begin{theorem}
		$W^{S\setminus\{\alpha_{n}\}}=\{ w_{n-\frac{1+(-1)^{h+1}}{2}}(l_{n-h})....w_{n-1}(l_{n-3})w_{n}(l_{n-2})w_{n-1}(l_{n-1})w_{n}(l_{n}):
		\  l_{n} < l_{n-1} < l_{n-2}<...<l_{n-h}\ and \ 
		\  0\leq h \leq n-2
		\}$
		\label{thm:desc}
		\end{theorem}
	\begin{lemma}
		\begin{enumerate}
			\item For $s=n$, define $p=\lceil\frac{n}{4}\rceil$. Then\\
			$w_{s,n}=w_{n}(n)w_{n-1}(n-1)w_{n}(n-2)w_{n-1}(n-3)...w_{n-1}(n-2p+3)w_{n}(n-2p+2)$
			\item For $s=n-1$, define $p=\lceil\frac{n-2}{4}\rceil$. Then\\
			$w_{s,n}=w_{n-1}(n)w_{n}(n-1)w_{n-1}(n-2)...w_{n-1}(n-2p+2)w_{n}(n-2p+1)$
			\item For $s=n-i$, ($n-1\geq i\geq2)$, define $p=\lceil\frac{n-i}{2}\rceil$. Then\begin{itemize}
				\item if p is even,\\
				$w_{s,n}=w_{n-1}(n-i+1)w_{n}(n-i)w_{n-1}(n-i-1)...w_{n-1}(n-i-p+3)w_{n}(n-i-p+2)$
				\item if p is odd,\\ $w_{s,n}=w_{n}(n-i+1)w_{n-1}(n-i)w_{n}(n-i-1)...w_{n-1}(n-i-p+3)w_{n}(n-i-p+2)$
			\end{itemize}
		\end{enumerate}	
		\label{type Dn min element}
	\end{lemma}
	\begin{proof}
		Recall from \cite[p. 69]{Hum1}, $\omega_{n}=\frac{1}{2}(\alpha_{1}+2\alpha_{2}+\cdots+(n-2)\alpha_{n-2}+\frac{1}{2}(n-2)\alpha_{n-1}+\frac{1}{2}n\alpha_{n})$.\\
		Proof of $(1):$ Let $w_{n,n}=w_{n-\frac{1+(-1)^{h+1}}{2}}(l_{n-h})....w_{n-1}(l_{n-3})w_{n}(l_{n-2})w_{n-1}(l_{n-1})w_{n}(l_{n})$ be satisfying the conditions  of the Theorem \ref{thm:desc}. From Lemma \ref{reduced expression} any reduced expression of $w_{n,n}$ begin with $s_{n}$. Hence $l_{n-h}=n$, and $h+1$ is odd. Note that $s_{n}$ does not appear in the $w_{n-1}$. Since $s_{n}$ must appear $p$ number of times, $h+1=2(p-1)+1$, i.e., $h=2(p-1)$. Now from first condition of Theorem \ref{thm:desc}, we see that to have the smalllest length element $l_{j-1}-l_{j}=1$, for $n-h+1\leq j\leq n$. Thus all $l_{j}$'s  are  determined recursively.\\
		Proof of $(2):$ Note that $s_{n-1}$ does not appear in $w_{n}$, then the argument is similar to $(1)$.\\
		Proof of $(3):$ Since $s\leq n-2$, for $l<n-i+1$, both $w_{n}(l)$ and $w_{n-1}(l)$  contains $s_{n-i}$.
		If $w_{s,n}=w_{n-\frac{1+(-1)^{h+1}}{2}}(l_{n-h})....w_{n-1}(l_{n-3})w_{n}(l_{n-2})w_{n-1}(l_{n-1})w_{n}(l_{n})$  is a reduced expression, then $h+1=p$ and $l_{n-h}=n-i+1$. The integer $\frac{1+(-1)^{h+1}}{2}$ will be determined by parity of $p$.	
	\end{proof}
	\begin{lemma}
		\begin{enumerate}
			\item For $s=n$, $X(w_{s,n})^{ss}_{\lambda_{s}}(\mathcal{L}(4\omega_{n}))=X(w_{s,n})^{s}_{\lambda_{s}}(\mathcal{L}(4\omega_{n}))$ if and only if $n\not\equiv0 \ (mod \ 4)$.
			\item For $s=n-1$, $X(w_{s,n})^{ss}_{\lambda_{s}}(\mathcal{L}(4\omega_{n}))=X(w_{s,n})^{s}_{\lambda_{s}}(\mathcal{L}(4\omega_{n}))$ if and only if $n\not\equiv 2 \ (mod\ 4)$.
			\item For $1\leq s\leq n-2$, $X(w_{s,n})^{ss}_{\lambda_{s}}(\mathcal{L}(4\omega_{n}))=X(w_{s,n})^{s}_{\lambda_{s}}(\mathcal{L}(4\omega_{n}))$ if and only if $s\not\equiv 0 \ (mod\ 2)$
		\end{enumerate}
		\label{ss=s,type Dn}
	\end{lemma}
	\begin{proof}
		proof of $(1):$ For $s=n$, $p=\lceil\frac{n}{4}\rceil$ we have \\ $w_{s,n}(\omega_{n})=\omega_{n}-\displaystyle\sum_{k=n-2p+1}^{n-2}(k-n+2p)\alpha_{k}-(p-1)\alpha_{n-1}-p\alpha_{n}$.  Then  $\langle w_{s,n}(\omega_{n}),\lambda_{s}\rangle=\frac{n}{4}-\lceil\frac{n}{4}\rceil$ is non zero if and only if  $n\not\equiv0 \ (mod \ 4)$. Now the proof follows from Lemma \ref{ss=s,min} $(ii)$.\\
		Proof of $(2) :$ For $s=n-1$, $p=\lceil\frac{n-2}{4}\rceil$ we have \\
		$w_{s,n}(\omega_{n})=\omega_{n}-\displaystyle\sum_{k=n-2p}^{n-2}(k-n+2p+1)\alpha_{k}-p\alpha_{n-1}-p\alpha_{n}$. Then  $\langle w_{s,n}(\omega_{n}),\lambda_{s}\rangle=\frac{n-2}{4}-\lceil\frac{n-2}{4}\rceil$ is non zero if and only if  $n\not\equiv2 \ (mod \ 4)$. Now the proof follows from Lemma \ref{ss=s,min} $(ii)$.\\
		Proof of $(3) :$ Let $s=n-i$ ($2\leq i\leq n-1$) and $p=\lceil\frac{n-i}{2}\rceil$. Then 
		\begin{itemize}
			\item if $p$ is odd, we have \\ $w_{s,n}(\omega_{n})=\omega_{n}-\displaystyle\sum_{k=n-i-p+1}^{n-i}(k-n+i+p)\alpha_{k}-\displaystyle\sum_{k=n-i+1}^{n-2}p\alpha_{k}-(\frac{p-1}{2})\alpha_{n-1}-(\frac{p+1}{2})\alpha_{n}$
			\item if $p$ is even, we have \\ $w_{s,n}(\omega_{n})=\omega_{n}-\displaystyle\sum_{k=n-i-p+1}^{n-i}(k-n+i+p)\alpha_{k}-\displaystyle\sum_{k=n-i+1}^{n-2}p\alpha_{k}-(\frac{p}{2})\alpha_{n-1}-(\frac{p}{2})\alpha_{n}$.\end{itemize}
		Then  $\langle w_{s,n}(\omega_{n}),\lambda_{s}\rangle=\frac{n-i}{2}-\lceil\frac{n-i}{2}\rceil$ is non zero if and only if  $n-i\not\equiv 0 \ (mod \ 2)$. Now the conclusion follows from Lemma \ref{ss=s,min}
		$(ii)$.
	\end{proof}
	\subsubsection{Description of $w_{s,n-1}$}	
	\begin{lemma}
		\begin{enumerate}
			\item For $s=n-1$, define $p=\lceil\frac{n}{4}\rceil$. Then\\			$w_{s,n-1}=w_{n-1}(n)w_{n}(n-1)w_{n-1}(n-2)w_{n}(n-3)...w_{n}(n-2p+3)w_{n-1}(n-2p+2)$
			\item For $s=n$, define $p=\lceil\frac{n-2}{4}\rceil$. Then\\
			$w_{s,n-1}=w_{n}(n)w_{n-1}(n-1)w_{n}(n-2)...w_{n}(n-2p+2)w_{n-1}(n-2p+1)$
			
			\item For s=n$-$i, ($n-1\geq i\geq2)$, define $p=\lceil\frac{n-i}{2}\rceil$. Then\begin{itemize}
				\item if p is even,\\
				$w_{s,n-1}=w_{n}(n-i+1)w_{n-1}(n-i)w_{n}(n-i-1)...w_{n}(n-i-p+3)w_{n-1}(n-i-p+2)$
				\item if p is odd,\\ $w_{s,n-1}=w_{n-1}(n-i+1)w_{n}(n-i)w_{n-1}(n-i-1)...w_{n}(n-i-p+3)w_{n-1}(n-i-p+2)$
			\end{itemize}
	\end{enumerate}	\end{lemma}
	\begin{proof}
		Recall from \cite[p.69]{Hum1}, $\omega_{n-1}=\frac{1}{2}(\alpha_{1}+2\alpha_{2}+\cdots+(n-2)\alpha_{n-2}+\frac{1}{2}n\alpha_{n-1}+\frac{1}{2}(n-2)\alpha_{n})$. Then the proof is similar to Lemma \ref{type Dn min element}.
	\end{proof}
	\begin{lemma}
		\begin{enumerate}
			\item For $s=n-1$, $X(w_{s,n-1})^{ss}_{\lambda_{s}}(\mathcal{L}(4\omega_{n-1}))=X(w_{s,n-1})^{s}_{\lambda_{s}}(\mathcal{L}(4\omega_{n-1}))$ if and only if $n\not\equiv0 \ (mod \ 4)$.
			\item For $s=n$, $X(w_{s,n-1})^{ss}_{\lambda_{s}}(\mathcal{L}(4\omega_{n-1}))=X(w_{s,n-1})^{s}_{\lambda_{s}}(\mathcal{L}(4\omega_{n-1}))$ if and only if $n\not\equiv 2 \ (mod\ 4)$.
			\item For $1\leq s\leq n-2$, $X(w_{s,n-1})^{ss}_{\lambda_{s}}(\mathcal{L}(4\omega_{n-1}))=X(w_{s,n-1})^{s}_{\lambda_{s}}(\mathcal{L}(4\omega_{n-1}))$ if and only if $s\not\equiv 0 \ (mod\ 2)$
		\end{enumerate}
	\end{lemma}
	\begin{proof}
		Proof is similar to that of  Lemma \ref{ss=s,type Dn}.
	\end{proof}
	\subsection{Type $E_{6}$}
	Note that $\omega_{1}$, $\omega_{6}$ are the only minuscule fundamental weights (see \cite[ p.180]{LW}).
	Recall from  \cite[p.69 ]{Hum1}, $\omega_{1}=\frac{1}{3}(4\alpha_{1}+3\alpha_{2}+5\alpha_{3}+6\alpha_{4}+4\alpha_{5}+2\alpha_{6})$ and $\omega_{6}=\frac{1}{3}(2\alpha_{1}+3\alpha_{2}+4\alpha_{3}+6\alpha_{4}+5\alpha_{5}+4\alpha_{6})$.
	\subsubsection{Description of $w_{s,1}$}
	\begin{lemma}
		\begin{enumerate}
			\item  $w_{1,1}=s_{1}s_{3}s_{4}s_{5}s_{2}s_{4}s_{3}s_{1}$ 
			
			\item $w_{2,1}=s_{2}s_{4}s_{3}s_{1}$ 
			
			\item  $w_{3,1}=s_{3}s_{4}s_{2}s_{5}s_{4}s_{3}s_{1}$
			
			\item  $w_{4,1}=s_{4}s_{5}s_{2}s_{4}s_{3}s_{1}$ 
			
			\item  $w_{5,1}=s_{5}s_{4}s_{6}s_{2}s_{5}s_{4}s_{3}s_{1}$

			\item $w_{6,1}=s_{6}s_{5}s_{4}s_{3}s_{1}$ 
		\end{enumerate}
		
	\end{lemma}	
	\begin{proof}
		Note that for $i=1,2,3,4$, $w_{1,1}=v_{i}w_{i,1}$ for some $v_{i}$, and $w_{5,1}=s_{5}s_{4}s_{2}w_{6,1}$. Therefore, it suffices to prove that $w_{1,1}, w_{5,1}\in W^{S\setminus \{\alpha_{1}\}}$ and they are reduced.
		\begin{itemize}
			\item	$w_{1,1}(\alpha_{2})=\alpha_{4}+\alpha_{5}+\alpha_{3}+\alpha_{1}$, $w_{1,1}(\alpha_{3})=\alpha_{3}$, $w_{1,1}(\alpha_{4})=\alpha_{4}$, $w_{1,1}(\alpha_{5})=\alpha_{2}$, $w_{1,1}(\alpha_{6})=\alpha_{5}+\alpha_{6}+\alpha_{4}+\alpha_{3}+\alpha_{1}$. Hence, $w_{1,1}(\alpha)>0$ for all $\alpha\in S\setminus \{\alpha_{1}\}$. This implies $w_{1,1}\in  W^{S\setminus \{\alpha_{1}\}}$. 	Further  we have $w_{1,1}(\omega_{1})=\omega_{1}-2\alpha_{1}-\alpha_{2}-2\alpha_{3}-2\alpha_{4}-\alpha_{5}$. Hence $ht(\omega_{1}-w_{1,1}(\omega_{1}))=8$. Since number of simple reflections involved in $w_{1,1}$ is also 8, using  \cite[Lemma 1.5, p.83]{s kannan} we conclude that $w_{1,1}$ is a reduced word.\item $w_{5,1}(\alpha_{2})=\alpha_{6}$, $w_{5,1}(\alpha_{3})=\alpha_{1}$, $w_{5,1}(\alpha_{4})=\alpha_{3}+\alpha_{4}+\alpha_{5}$, $w_{5,1}(\alpha_{5})=\alpha_{2}$, $w_{5,1}(\alpha_{6})=
			\alpha_{4}$. Therefore, $w_{5,1}\in  W^{S\setminus \{\alpha_{1}\}}$. We have 
			$w_{5,1}(\omega_{1})=\omega_{1}-\alpha_{1}-\alpha_{2}-\alpha_{3}-2\alpha_{4}-2\alpha_{5}-\alpha_{6}$. Hence $ht(\omega_{1}-w_{5,1}(\omega_{1}))=8$. Then the similar argument shows that the expression above for $w_{5,1}$ is a reduced.\end{itemize}\end{proof}
	\begin{lemma}
		$X(w_{s,r})^{ss}_{\lambda_{s}}(\mathcal{L}(3\omega_{1}))=X(w_{s,r})^{s}_{\lambda_{s}}(\mathcal{L}(3\omega_{1}))$ if and only if $s\neq 2,4$
	\end{lemma}
	\begin{proof}
		From the description of $\omega_{1}$, we have $\langle w_{s,1}(\omega_{1}),\lambda_{s}\rangle$ is zero if and only if $s=2,4$ Then the proof follows from the criterion proved in Lemma \ref{ss=s,min} $(ii)$.
	\end{proof}
	\subsubsection{Description of $w_{s,6}$}
	\begin{lemma}
		\begin{enumerate}
			\item  $w_{1,6}=s_{1}s_{3}s_{4}s_{5}s_{6}$ 
			\item $w_{2,6}=s_{2}s_{4}s_{5}s_{6}$ 
			\item  $w_{3,6}=s_{3}s_{4}s_{2}s_{1}s_{3}s_{4}s_{5}s_{6}$ 
			\item $w_{4,6}=s_{4}s_{3}s_{2}s_{4}s_{5}s_{6}$
			\item $w_{5,6}=s_{5}s_{4}s_{2}s_{3}s_{4}s_{5}s_{6}$ 
			\item $w_{6,6}=s_{6}s_{5}s_{4}s_{3}s_{2}s_{4}s_{5}s_{6}$ 
		\end{enumerate}
	\end{lemma}
	\begin{proof}
		We have $w_{3,6}=s_{3}s_{4}s_{2}w_{1,6}$ and for $i=2,4,5,6$, $w_{6,6}=v_{i}w_{i,6}$ for some $v_{i}$. Therefore, it suffices to show that $w_{3,6}, w_{6,6}\in W^{S\setminus \{\alpha_{6}\}}$ and they are reduced.\begin{itemize}
			\item
			$w_{3,6}(\alpha_{1})=\alpha_{4}$, $w_{3,6}(\alpha_{2})=\alpha_{1}$, $w_{3,6}(\alpha_{3})=\alpha_{2}$, $w_{3,6}(\alpha_{4})=\alpha_{2}$, $w_{3,6}(\alpha_{5})=\alpha_{6}$. This shows that $w_{3,6}(\alpha)>0$  for all $\alpha\in S\setminus\{\alpha_{6}\}$. Hence $w_{3,6}\in W^{S\setminus \{\alpha_{6}\}}$ . We also have that $w_{3,6}(\omega_{6})=\omega_{6}-\alpha_{6}-\alpha_{5}-2\alpha_{4}-2\alpha_{3}-\alpha_{2}-\alpha_{1}$. Hence, $ht(\omega_{6}-w_{3,6}(\omega_{6}))=8$. Since the number of simple reflection appearing in the expression of $w_{3,6}$ is also 8, \cite[Lemma 1.5, p.83]{s kannan} implies that the word is reduced.\item
			$w_{6,6}(\alpha_{1})=\alpha_{1}+\alpha_{3}+\alpha_{4}+\alpha_{5}+\alpha_{6}$, $w_{6,6}(\alpha_{2})(\alpha_{3})$, $w_{6,6}(\alpha_{3})=\alpha_{2}$, $w_{6,6}(\alpha_{4})=\alpha_{4}$, $w_{6,6}(\alpha_{5})=\alpha_{5}$. Moreover, $w_{6,6}(\omega_{6})=\omega_{6}-2\alpha_{6}-2\alpha_{5}-2\alpha_{4}-\alpha_{2}-\alpha_{3}$. Hence $ht(\omega_{6}-w_{6,6}(\omega_{6})=8$. Then the similar argument as above shows that $w_{6,6}\in W^{S\setminus \{\alpha_{6}\}}$ and the expression of $w_{6,6}$ is reduced.
		\end{itemize}
	\end{proof}	
	\begin{lemma}	$X(w_{s,r})^{ss}_{\lambda_{s}}(\mathcal{L}(3\omega_{6}))=X(w_{s,r})^{s}_{\lambda_{s}}(\mathcal{L}(3\omega_{6}))$ if and only if $s\neq 2,4$
	\end{lemma}
	\begin{proof}
		Using the expression of $\omega_{6}$ in terms of simply roots, we see $\langle w_{s,6}(\omega_{6}),\lambda_{s}\rangle$ is zero if and only if $s=2,4$ Then the proof follows from the Lemma \ref{ss=s,min}$(ii)$.
	\end{proof}
	\subsection{Type $E_{7}$}
	Note that the only minuscule fundamental weight in type $E_{7}$ is $\omega_{7}$ (see \cite[ p.180]{LW}). Recall from \cite[p.69 ]{Hum1}, $\omega_{7}=\frac{1}{2}(2\alpha_{1}+3\alpha_{2}+4\alpha_{3}+6\alpha_{4}+5\alpha_{5}+4\alpha_{6}+3\alpha_{7})$.
	\begin{lemma}
		\begin{enumerate}
			\item  $w_{1,7}=s_{1}s_{3}s_{4}s_{5}s_{6}s_{7}$	
			\item  $w_{2,7}=s_{2}s_{4}s_{5}s_{3}s_{4}s_{1}s_{2}s_{3}s_{4}s_{5}s_{6}s_{7}$ .		
			\item  $w_{3,7}=s_{3}s_{4}s_{1}s_{2}s_{3}s_{4}s_{5}s_{6}s_{7}$ 	
			\item $w_{4,7}=s_{4}s_{3}s_{5}s_{4}s_{1}s_{2}s_{3}s_{4}s_{5}s_{6}s_{7}$ 				
			\item  $w_{5,7}=s_{5}s_{6}s_{4}s_{3}s_{5}s_{4}s_{2}s_{1}s_{3}s_{4}s_{5}s_{6}s_{7}$
			\item  $w_{6,7}=s_{6}s_{5}s_{4}s_{3}s_{2}s_{4}s_{5}s_{6}s_{7}$
			\item  $w_{7,7}=s_{7}s_{6}s_{5}s_{4}s_{3}s_{2}s_{4}s_{5}s_{6}s_{7}$
		\end{enumerate}
	\end{lemma}		
	\begin{proof}
		For $i=1,2,3,4$, we have $w_{2,7}=v_{i}w_{i,7}$ for some $v_{i}$ and  $w_{7,7}=s_{7}w_{6,7}$. Therefore, it suffices to show that $w_{2,7},w_{5,7},w_{7,7}\in W^{S\setminus \{\alpha_{7}\}}$ and they are also reduced.	\begin{itemize}
			\item 	$w_{2,7}(\alpha_{1})=\alpha_{5}$, $w_{2,7}(\alpha_{2})=\alpha_{1}$, $w_{2,7}(\alpha_{3})=\alpha_{4}$, $w_{2,7}(\alpha_{4})=\alpha_{3}$, $w_{2,7}(\alpha_{5})=\alpha_{6}+\alpha_{5}+\alpha_{4}+\alpha_{2}$, $w_{2,7}(\alpha_{6})=\alpha_{7}$. Therefore, $w_{2,7}\in W^{S\setminus \{\alpha_{7}\}}$.\\  $w_{2,7}(\omega_{7})=\omega_{7}-\alpha_{7}-\alpha_{6}-2\alpha_{5}-3\alpha_{4}-2\alpha_{3}-2\alpha_{2}-\alpha_{1}$ and $ht(\omega_{7}-w_{2,7}(\omega_{7})=12$.
			\item $w_{5,7}(\alpha_{1})=\alpha_{6}$, $w_{5,7}(\alpha_{2})=\alpha_{1}$, $w_{5,7}(\alpha_{3})=\alpha_{2}+\alpha_{4}+\alpha_{5}$, $w_{5,7}(\alpha_{4})=\alpha_{3}$, $w_{5,7}(\alpha_{5})=\alpha_{4}$, $w_{5,7}(\alpha_{6})=\alpha_{7}+\alpha_{6}+\alpha_{5}$. Therefore, $w_{5,7}\in W^{S\setminus \{\alpha_{7}\}}$.\\
			$w_{5,7}(\omega_{7})=\omega_{7}-\alpha_{7}-2\alpha_{6}-3\alpha_{5}-3\alpha_{4}-2\alpha_{3}-\alpha_{1}-\alpha_{2}$ and $ht(\omega_{7}-w_{5,7}(\omega_{7})=13$.
			\item $w_{7,7}(\alpha_{1})=\alpha_{1}+\alpha_{3}+\alpha_{4}+\alpha_{5}+\alpha_{6}+\alpha_{7}$, $w_{7,7}(\alpha_{2})=\alpha_{3}$, $w_{7,7}(\alpha_{3})=\alpha_{2}$, $w_{7,7}(\alpha_{4})=\alpha_{4}$, $w_{7,7}(\alpha_{5})=\alpha_{5}$, $w_{7,7}(\alpha_{6})=\alpha_{6}$. Therefore, $w_{7,7}\in W^{S\setminus \{\alpha_{7}\}}$.\\
			$w_{7,7}(\omega_{7})=\omega_{7}-2\alpha_{7}-2\alpha_{6}-2\alpha_{5}-2\alpha_{4}-\alpha_{2}-\alpha_{3}$ and $ht(\omega_{7}-w_{7,7}(\omega_{7})=10$
		\end{itemize}
		The above computations and  \cite[Lemma 1.5, p.83]{s kannan} imply that  $w_{2,7}$, $w_{5,7}$, and $w_{7,7}$ are reduced word. 
	\end{proof}
	\begin{lemma}
		$X(w_{s,r})^{ss}_{\lambda_{s}}(\mathcal{L}(2\omega_{7}))=X(w_{s,r})^{s}_{\lambda_{s}}(\mathcal{L}(2\omega_{7}))$ if and only if $s\neq 1,3,4,6$
	\end{lemma}
	\begin{proof}
		From the description of $w_{7}$, we see that $\langle w_{s,7}(\omega_{7}),\lambda_{s}\rangle=0$ if and only if $s=1,3,4,6$. Then the proof follows from Lemma \ref{ss=s,min}$(ii)$.
	\end{proof}
	\section{GIT quotients of $X(w_{s,r})^{ss}_{\lambda_{s}}(\mathcal{L})$}
	In this section, we prove that if $\omega_{r}$ is minuscule and $\lambda_{s}$ is cominuscule such that\\ $X(w_{s,r})^{ss}_{\lambda_{s}}(\mathcal{L}(m\omega_{r}))=X(w_{s,r})^{s}_{\lambda_{s}}(\mathcal{L}(m\omega_{r}))$, then $\opsgmodx{m}$ is a projective space.
	\begin{lemma}
		Let $P$, $\chi$ and $J$ be as in Section 2. Let $w\in W^{J}$ be a minimal element with the property that $\langle w(\chi),\lambda_{s}\rangle=0$. Then $\GmodX{\lambda_{s}}{X(w)}{\mathcal{L}(\chi)}$ is a point.\label{quotient is point}
	\end{lemma}
	\begin{proof} 
		Let $p\in H^{0}(X(w),\mathcal{L}(\chi))$ be an $\lambda_{s}$-invariant section of weight $\mu$. 
		Then $\mu=-\sum_{i=2}^{k}c_{i}\tau_{i}(\chi)-c_{1}w(\chi)$, where $\tau_{i}\in W^{J}$, $\tau_{i}<w$ for $i\neq 1$ and $c_{i}\in \mathbb{Q}_{\geq 0}$, $\sum_{i=1}^{k}c_{i}=1$. Since $p$ is $\lambda_{s}$-invariant, we have $\langle\mu,\lambda_{s}\rangle=0$. Therefore, \begin{equation}
			\sum_{i=2}^{k}c_{i}\langle \tau_{i}(\chi),\lambda_{s}\rangle+c_{1}\langle w(\chi),\lambda_{s}\rangle=0
		\end{equation} By hypothesis $\langle w(\chi),\lambda_{s}\rangle=0$. By Lemma \ref{main lemma}, for $\tau_{i}<w$, we have $\langle \tau_{i}(\chi),\lambda_{s}\rangle>0$ for all $2\leq i\leq k$. Hence, from equation (6.1), we have $c_{i}=0$ for all $2\leq i\leq k$. Thus $c_{1}=1$ and $\mu=-w(\chi)$.  Hence, $p_{w}$ is the only (up to non zero scalar) $\lambda_{s}$-invariant section. Therefore, $\GmodX{\lambda_{s}}{X(w)}{\mathcal{L}(\chi)}=Proj(\mathbb{C}[p_{w}])$ consists of a single point.
	\end{proof}
	Now onwards till the end of this section, we assume that $\omega_{r}$ is minuscule.
	\begin{corollary}
		Let $1\leq r,s\leq n$ and $\mathcal{L}=\mathcal{L}(m\omega_{r})$. Assume that $X(w_{s,r})^{ss}_{\lambda_{s}}(\mathcal{L})\neq X(w_{s,r})^{s}_{\lambda_{s}}(\mathcal{L})$, then the quotient $\GmodX{\lambda_{s}}{X(w_{s,r})}{\mathcal{L}}$ is a point.
		\label{qt is point, special case}
	\end{corollary}
	\begin{proof}
		From Lemma \ref{ss=s,min}$(ii)$, we have $\langle w_{s,r}(\omega_{r}),\lambda_{s}\rangle=0$. Then, by Lemma \ref{quotient is point}, $\GmodX{\lambda_{s}}{X(w_{s,r})}{\mathcal{L}}$ is a point.
	\end{proof}
	\begin{lemma}
		Let $1\leq r,s\leq n$ and $\mathcal{L}=\mathcal{L}(m\omega_{r})$. If $\alpha_{s}$ is cominuscule and $X(w_{s,r})^{ss}_{\lambda_{s}}(\mathcal{L})= X(w_{s,r})^{s}_{\lambda_{s}}(\mathcal{L})$. Then 
		$\GmodX{\lambda_{s}}{X(w_{s,r})}{\mathcal{L}}$ is smooth.
		\label{smoothness in cominuscule}
	\end{lemma}
	\begin{proof}
		Since $X(w_{s,r})$  is the minimal dimensional Schubert variety admitting semistable point for $\lambda_{s}$, from Bruhat decomposition it follows that $X(w_{s,r})^{ss}_{\lambda_{s}}(\mathcal{L})\subseteq Bw_{s,r}P/P$. Thus, $X(w_{s,r})^{ss}_{\lambda_{s}}(\mathcal{L})$ is smooth open subset of $X(w_{s,r})$.			
		
		Since 	$X(w_{s,r})^{ss}_{\lambda_{s}}(\mathcal{L})=X(w_{s,r})^{s}_{\lambda_{s}}(\mathcal{L})$, for any point $x\in X(w_{s,r})^{ss}_{\lambda_{s}}(\mathcal{L})$, the $\lambda_{s}$-orbit in $X(w_{s,r})_{\lambda_{s}}^{ss}(\mathcal{L})$ is closed and the stabilizer of $x$ in $\lambda_{s}(G_{m})$ is finite.\\
		Choose a subset $\{\beta_{1},\beta_{2},\ldots,\beta_{k}\}$ of positive roots in $R^{+}(w_{s,r}^{-1})$ such that\[x=u_{\beta_{1}}(t_{1})u_{\beta_{2}}(t_{2})\cdots u_{\beta_{k}}(t_{k})w_{s,r}P/P\]
		with $u_{\beta_{j}}(t_{j})\in U_{\beta_{j}}$ and $t_{j}\neq0$ (see \cite[Theorem 28.4]{Hum2}). For $a\in G_{m}$, we have \[\lambda_{s}(a)\cdot x=u_{\beta_{1}}(a^{\langle \beta_{1},\lambda_{s}\rangle}t_{1})u_{\beta_{2}}(a^{\langle \beta_{2},\lambda_{s}\rangle}t_{2})\cdots u_{\beta_{k}}(a^{\langle \beta_{k},\lambda_{s}\rangle}t_{k})w_{s,r}P/P\]
		If $\langle \beta_{j}, \lambda_{s}\rangle=0$ for all $1\leq j\leq k$, then $a^{\langle \beta_{j}, \lambda_{s}\rangle}=1$ $\forall a\in G_{m}$. Thus $\lambda_{s}(a)\cdot x=x$, $\forall a\in G_{m}$. But stabilizer of $x$ in $\lambda_{s}(G_{m})$ is finite. Therefore, there is an integer $1\leq j\leq k$ such that $\langle \beta_{j}, \lambda_{s}\rangle\neq0$. Since $\alpha_{s}$ is cominuscule, we have $\langle \beta_{j}, \lambda_{s}\rangle=1$.\\
		Now, let $a\in G_{m}$ be such that $\lambda_{s}(a)\cdot x=x$. Then $a^{\langle \beta_{j}, \lambda_{s}\rangle}=1$ and hence $a=1$.\\
		Therefore, the action of $\lambda_{s}(G_{m})$ on $X(w_{s,r})_{\lambda_{s}}^{ss}(\mathcal{L})$ is free, and the quotient is geometric. Then the results follows from Luna's slice theorem (see \cite[Proposition 5.7]{Luna}).
	\end{proof}
	\subsection{Proof of Proposition 1.2} Let  $1\leq s, r\leq n$ be integers such that $\omega_{r}$ is minuscule and $\alpha_{s}$ is cominuscule. We write  $w=w_{s,r}$ and $\mathcal{L}=\mathcal{L}(\omega_{r})$. Let $k=l(w)$. \\
	Let $U_{w}:=\displaystyle\Pi_{\beta\in R^{+}(w^{-1})}U_{\beta}$. Note that the map $\psi:U_{w}\longrightarrow BwP/P\subseteq X(w)$ defined by $\psi(u)=u\cdot wP/P$ is an isomorphism of varieties. In particular, $BwP/P$ is isomorphic to the affine space $\mathbb{C}^{l(w)}$(see \cite[Ch 13, p. 385]{Jan}). Also note that $BwP/P=U_{w}wP/P$. Further, for any $t\in T$, $\psi(tut^{-1})=t\cdot \psi(u)$ $\forall t\in T, u\in U_{w}$.
	\\We denote the restriction of $\mathcal{L}$ to $BwP/P$ also by $\mathcal{L}$.
	
	Since $U_{w}$ is an affine space, $\mathbb{C}[U_{w}]$ is a unique factoriazation domain, Hence, $Pic(\mathbb{C}[U_{w}])$ is trivial (see \cite [Example 6.3.1, p. 132]{Har}). Therefore, $\psi^{*}(\mathcal{L})$ is the  trivial line bundle on $U_{w}$.
	Thus, we have $\psi^{*}(H^{0}(X(w),\mathcal{L}))$ is a subspace of $\mathbb{C}[U_{w}]$.\\
	Let $B_{w}:=TU_{w}$. Consider the acion of $B_{w}$ on $U_{w}$ by $th\cdot u=thut^{-1}$ $\forall t\in T$ and $h,u \in U_{w}$.
	Then $\psi$ is $B_{w}$ - equivariant, where the action of $B_{w}$ on $BwP/P$ is given by the left multiplication.\\
	Let $X_{\beta}$ be the coordinate function on $U_{w}$ corresponding to $U_{\beta}$, for $\beta\in R^{+}(w^{-1})$. That is, $X_{\beta}(\Pi_{\gamma\in R^{+}(w^{-1})}u_{\gamma}(a_{\gamma}))=a_{\beta}$ for  $(a_{\gamma})_{\gamma\in R^{+}(w^{-1})}\in \mathbb{C}^{k}$.
	\begin{remark}
		$X_{\beta}$'s does not depend on the ordering of $R^{+}(w^{-1})$.
	\end{remark}
	\begin{proof}
		Let $R^{+}(w^{-1})=\{\beta_{1},\ldots,\beta_{k}\}$.  Since $w^{-1}\in W^{S\setminus\{\alpha_{s}\}}$, $\alpha_{s}\leq \beta_{i}$ for $1\leq i\leq k$. Since $\alpha_{s}$ is cominuscule, $\beta_{i}+\beta_{j}$ is not a root for $1\leq i\neq j\leq k$. This implies $u_{\beta_{i}}(x_{\beta_{i}})u_{\beta_{j}}(x_{\beta_{j}})=u_{\beta_{j}}(x_{\beta_{j}})u_{\beta_{i}}(x_{\beta_{i}})$ for all $i\neq j$ and $x_{\beta_{i}}$, $x_{\beta_{j}}\in \mathbb{C}$ (see \cite[Proposition 8.2.3]{Springer}). Therefore, the product $\displaystyle\Pi_{\beta\in R^{+}(w^{-1})}u_{\beta}(x_{\beta})$ does not depend on the ordering of $\beta_{i}$'s.
	\end{proof}
	\begin{lemma}
		We have \begin{enumerate}
			\item $\psi^{*}(p_{w})$ is a non zero constant.
			\item For any $\beta\in R^{+}(w^{-1})$, $\psi^{*}(p_{s_{\beta}w})\in \mathbb{C}^{\times}\cdot X_{\beta}$.
		\end{enumerate}
		\label{restriction of line bundle to cell}
	\end{lemma}
	\begin{proof}
		Proof of $(1):$ Note that the $p_{w}$ is a $T$-weight vector of weight  $-w(\omega_{r})$ . Further, any weight $\mu$ of $H^{0}(X(w),\mathcal{L}(\omega_{r}))$ satisfies $\mu\leq -w(\omega_{r})$ (see  \cite[Theorem 2.2, p. 168]{Seshadri}). The one dimensional subspace $\mathbb{C}\cdot p_{w}$ is $B$ - stable and the associated character is $-w(\omega_{r})$. Hence, $\mathbb{C}\cdot p_{w}$ is $B_{w}$- stable as well. Since $\psi$ is $B_{w}$ - equivariant, $\psi^{*}(H^{0}(X(w),\mathcal{L}(\omega_{r})))$  is a $B_{w}$-submodule of $\mathbb{C}[U_{w}]$. In particular, $\psi^{*}(p_{w})$ is $U_{w}$-invariant. Therefore, $\psi^{*}(p_{w})\in \mathbb{C}[U_{w}]^{U_{w}}=\mathbb{C}\cdot 1=$ the vector subspace of constant functions on $U_{w}$. \\	Proof of $(2):$ Let $R^{+}(w^{-1})=\{\beta_{1},\ldots,\beta_{k}\}$. 
		Fix $1\leq l\leq k$. Let $f\in \mathbb{C}[U_{w}]\setminus\{0\}$ be a $T$-weight vector of weight $-\beta_{l}$. Write $f=\displaystyle\sum_{(i_{1},\ldots,i_{k})\in \mathbb{Z}^{k}_{\ge 0}}a_{\underline{i}}X_{\beta_{1}}^{i_{1}}\cdots X_{\beta_{k}}^{i_{k}}$, $a_{\underline{i}}\in \mathbb{C}$. Note that $T$-weight of the monomial $X_{\beta_{1}}^{i_{1}}\cdots X_{\beta_{k}}^{i_{k}}$ is $-\displaystyle\sum_{j=1}^{k}i_{j}\beta_{j}$. Hence, for all  $a_{\underline{i}}\neq 0$, we have \begin{equation}
			\displaystyle\sum_{j=1}^{k} i_{j}\beta_{j}=\beta_{l}
		\end{equation}
		Since, $\beta_{j}\geq \alpha_{s}$ and  $\alpha_{s}$ is cominuscule, coefficient of $\alpha_{s}$ in the expression of $\beta_{j}$ is 1 for all $1\leq j\leq k$. Thus, $\langle \beta_{j},\lambda_{s}\rangle=1$ for all $1\leq j\leq k$. Therefore, from equation (6.2) $\langle\displaystyle\sum_{j=1}^{k}i_{j}\beta_{j},\lambda_{s}\rangle=\langle\beta_{l},\lambda_{s}\rangle$. 	Hence, for any 
		$\underline{i}\in \mathbb{Z}^{k}_{\ge 0}$ such that $a_{\underline{i}}\neq 0$, $\displaystyle\sum_{j=1}^{k} i_{j}=1$. Thus, $i_{j_{0}}=1$ for exactly one $j_{0}$ and $i_{j}=0$ for $j\neq j_{0}$. Therefore, $f=\displaystyle\sum_{i=1}^{k}a_{i}X_{\beta_{i}}$, $a_{i}\in \mathbb{C}$.
		Since $X_{\beta_{i}}$'s are weight vectors of distinct weight spaces,
		$\{X_{\beta_{i}}: 1\leq i\leq k\}$ is a linearly independent subset of $\mathbb{C}[U_{w}]$. Therefore, $f=c\cdot X_{\beta_{l}}$ for some $c\in \mathbb{C}^{\times}$.\\ 
		Let $\beta\in R^{+}(w^{-1})$. Since $p_{s_{\beta}w}$ is a $T$-weight vector in  $H^{0}(X(w),\mathcal{L}(\omega_{r}))$ of weight $-\beta-w(\omega_{r})$, $\psi^{*}(p_{s_{\beta}w})$ is a weight vector in $\mathbb{C}[U_{w}]$ of weight $-\beta$. Thus, we have $\psi^{*}(p_{s_{\beta}w})=c\cdot X_{\beta}$ for some $c\in \mathbb{C}^{\times}$	  
	\end{proof}
	\begin{proposition}
		Let $1\leq r,s\leq n$ and assume that $\omega_{r}$ is minuscule and $\alpha_{s}$ is cominuscule. Let $m\in \mathbb{N}$ be the least positive integer such that $\langle m\omega_{r},\lambda_{s}\rangle\in \mathbb{N}$, and let $a=-m\langle w(\omega_{r}),\lambda_{s}\rangle$. Assume that $X(w)^{ss}_{\lambda_{s}}(\mathcal{L}(m\omega_{r}))=X(w)^{s}_{\lambda_{s}}(\mathcal{L}(m\omega_{r}))$. Then, we have $\GmodX{\lambda_{s}}{X(w_{s,r})}{\mathcal{L}}\simeq (\mathbb{P}^{k-1},\mathbb{O}(a))$.
		\label{quotient is projective space}
	\end{proposition}
	\begin{proof}
		Let $a_{1}=a$, $a_{2}=m+\langle w(m\omega_{r}),\lambda_{s}\rangle$ and $\mathcal{L}=\mathcal{L}(m\omega_{r})$. Since  $X(w)^{ss}_{\lambda_{s}}(\mathcal{L})=X(w)^{s}_{\lambda_{s}}(\mathcal{L})$,
		from Lemma \ref{ss=s,min} $(i)$, we have $-m<\langle w(m\omega_{r}),\lambda_{s}\rangle<0$. This implies $a_{1},a_{2}>0$. Therefore, $\lambda_{s}$-weight of $p_{{s_{\beta}}w}^{a_{1}}p_{w}^{a_{2}}$ is  	
		\begin{equation}
			-\langle(a_{1}s_{\beta}w(\omega_{r})+a_{2}w(\omega_{r})),\lambda_{s}\rangle
			=-\langle a_{1}(w(\omega_{r})+\beta)+a_{2}w(\omega_{r}),\lambda_{s}\rangle
			=-(a_{1}+a_{2})\langle w(\omega_{r}),\lambda_{s}\rangle-a_{1},
		\end{equation} which is equal to zero.
		Let $d\in \mathbb{Z}_{\geq 0}$ and $(m_{1},\ldots,m_{k+1})\in \mathbb{Z}_{\geq 0}^{k+1}$ such that $\displaystyle\sum_{i=1}^{k+1}m_{i}=dm$. Then by equation (6.3), $\displaystyle\prod_{i=1}^{k}p_{s_{\beta_{i}}w}^{m_{i}}p_{w}^{m_{k+1}}\in H^{0}(X(w),\mathcal{L}(dm\omega_{r}))^{\lambda_{s}}$ if and only if $\displaystyle\sum_{i=1}^{k}m_{i}=da_{1}$ and $m_{k+1}=da_{2}$.	\\
		Hence, we have $\psi^{*}(H^{0}(X(w),\mathcal{L}(dm\omega_{r}))^{\lambda_{s}})=\mathbb{C}[U_{w}]_{da_{1}}$= space of all homogeneous forms of degree $da_{1}$ in $X_{\beta_{1}},\ldots,X_{\beta_{k}}$. Further the following diagram is commutative :\\
		\begin{tikzcd}
			H^{0}(X(w),\mathcal{L}(d_{1}m\omega_{r}))^{\lambda_{s}}\otimes H^{0}(X(w),\mathcal{L}(d_{2}m\omega_{r}))^{\lambda_{s}} \arrow[r] \arrow[d, "\psi^{*}\otimes\psi^{*}"]
			& H^{0}(X(w),\mathcal{L}((d_{1}+d_{2})m\omega_{r}))^{\lambda_{s}} \arrow[d, "\psi^{*}"] \\
			\mathbb{C}[U_{w}] \otimes \mathbb{C}[U_{w}]\arrow[r]
			&  \mathbb{C}[U_{w}]
		\end{tikzcd}
		
		where the horizontal arrows are multiplication maps and $d_{1},d_{2}\in \mathbb{Z}_{\geq 0}$.

		Let $V=\displaystyle\sum_{i=1}^{k}\mathbb{C}X_{\beta_{i}}\subseteq \mathbb{C}[U_{w}]$. Thus,
		
		\begin{tikzcd}
			\displaystyle\bigoplus_{d\in \mathbb{Z}_{\geq 0}}H^{0}(X(w),\mathcal{L}(dm\omega_{r}))^{\lambda_{s}}\arrow[r,"\psi^{*}"] &  \displaystyle\bigoplus_{d\in \mathbb{Z}_{\geq 0}}\mathbb{C}[U_{w}]_{da_{1}}=\bigoplus_{d\in \mathbb{Z}_{\geq 0}}Sym^{da_{1}}(V)
			& 
		\end{tikzcd} 
		
		is an isomorphism of graded $\mathbb{C}$-algebras.
		Now, the proof of the proposition follows from the fact that $Proj(\bigoplus_{d\in \mathbb{Z}_{\geq 0}}Sym^{da_{1}}(V))=(\mathbb{P}^{k-1},\mathbb{O}(a_{1}))$.
	\end{proof}
	\section{GIT quotient of $X(w_{{s,r}})$ in $G_{r,n}$}
	In this section, we prove Theorem \ref{main result} and give a decomposition of $H^{0}(X(w_{s,r}),\mathcal{L}(km\omega_{r}))^{n\cdot\lambda_{s}}$ into irreducible $SL(s-p)\times SL(r-p)$-modules for all $k$. 
	\\
	Let $\hat{G}=SL(n,\mathbb{C})$ and $\hat{T}$ denotes maximal torus consisting of diagonal matrices in $\hat{G}$. For $1\leq s\leq n-1$ and $t\in G_{m}$, we have  \[(n\cdot\lambda_{s})(t)=diag(t^{n-s},\ldots, t^{n-s},t^{-s},t^{-s},\ldots,t^{-s})\] 
	(first $s$ many are $t^{n-s}$).\\
	Let $w_{s,r}$ and $p$ as in Lemma \ref{criterion An}. In one line notation, we have \[w_{s,r}=(1,2,\ldots,p-1,p,s+1,s+2,\ldots,s+r-p)\]
	Let $w=w_{s,r}$. Recall that the stabilizer of $X(w)$ in $\hat{G}$ is the parabolic subgroup $P_{w}=P_{S\setminus\{\alpha_{p},\alpha_{s+r-p}\}}$.
	The semisimple part of the Levi subgroup of $C_{P_{w}}(n\cdot \lambda_{s})$ is isomorphic to $SL(s-p,\mathbb{C})\times SL(r-p,\mathbb{C})$.
	Let $G^{\prime}=SL(s-p)\times SL(r-p)$.  
	Let $B_{1}$ (respectively, $B_{2}$) be the Borel subgroup consisting of upper triangular matrices in $SL(s-p,\mathbb{C})$ (respectively, $SL(r-p,\mathbb{C})$). Let $T_{1}$ (respectively, $T_{2}$) be the maximal torus consisting of diagonal matrices in $SL(s-p,\mathbb{C})$ (respectively, $SL(r-p,\mathbb{C})$). Then, $B_{1}\times B_{2}$ is a Borel subgroup of $SL(s-p,\mathbb{C})\times SL(r-p,\mathbb{C})$ containing $T_{1}\times T_{2}$.
	Let $c=\frac{rs}{(rs,n)}$,  $m=\frac{n}{(rs,n)}$, where $(rs,n)$ denotes the GCD of $rs$ and $n$. Also let $a=c-mp$.
	\begin{theorem} Assume that $1\leq r,s\leq n-1$, and $n\nmid rs$. Then we have\\
		$\GmodX{n\cdot\lambda_{s}}{X(w_{s,r})}{\mathcal{L}(m\omega_{r})}=(\mathbb{P}(M(s-p,r-p)) ,\mathbb{O}(a))$.
		\label{git quotient type An}
	\end{theorem}
	
	\begin{proof}
		\textbf{Step 1: }$SL(s-p)\times 1$ and $1\times SL(r-p)$ does not fix any point in $X(w_{s,r})^{ss}_{n\cdot\lambda_{s}}(\mathcal{L}(m\omega_{r}))$.\\ Let $w=w_{s,r}$. Let $u\in U_{w}\setminus\{1\}$. Let $R_{1}^{+}=\{\beta\in R^{+}(w^{-1}): \beta\geq \alpha_{s-1}\}$. If $u\in \prod_{\beta\in R^{+}(w^{-1})\setminus R_{1}^{+}}U_{\beta}$, then $u_{\alpha_{s-1}}(1)$ does not fix $uwP/P$. If  $u\notin \prod_{\beta\in R^{+}(w^{-1})\setminus R_{1}^{+}}U_{\beta}$, then there exists $t\in T_{1}$ such that  $tuwP/P\neq uwP/P$. Therefore, $SL(s-p)\times 1$ does not fix any point in $X(w_{s,r})^{ss}_{n\cdot\lambda_{s}}(\mathcal{L}(m\omega_{r}))$. Similarly, $1\times SL(r-p)$ does not fix any point in $X(w_{s,r})^{ss}_{n\cdot\lambda_{s}}(\mathcal{L}(m\omega_{r}))$.\\ Let  $V_{1}\otimes V_{2}\subseteq H^{0}(\mathbb{P}^{k-1},\mathbb{O}(1))$ be an irreducible $SL(s-p)\times SL(r-p)$-module.
		
		\textbf{Step 2:} Both $V_{1}$ and $V_{2}$ are non trivial. If $V_{1}$ is trivial, then since
		$\GmodX{n\cdot\lambda_{s}}{X(w_{s,r})}{\mathcal{L}(m\omega_{r})}=(\mathbb{P}^{k-1},\mathbb{O}(a))$, $SL(s-p)\times 1$ fixes a point $y$ in $\GmodX{n\cdot\lambda_{s}}{X(w_{s,r})}{\mathcal{L}(m\omega_{r})}$. \\Let $\pi:X(w_{s,r})^{ss}_{n\cdot\lambda_{s}}(\mathcal{L}(m\omega_{r}))\longrightarrow \GmodX{n\cdot\lambda_{s}}{X(w_{s,r})}{\mathcal{L}(m\omega_{r})}$ be the GIT morphism. Let $x\in X(w_{s,r})^{ss}_{n\cdot\lambda_{s}}(\mathcal{L}(m\omega_{r}))$ be such that $\pi(x)=y$. Then \begin{equation}
			(SL(s-p)\times \{1\})\cdot x\subseteq \pi^{-1}(\{y\})= (n\cdot\lambda_{s}(G_{m}))\cdot x
		\end{equation}
		Let $H:=n\cdot\lambda_{s}(G_{m})/(n\cdot\lambda_{s}(G_{m}))\cap Z(SL(n,\mathbb{C}))$. Since  $n\cdot\lambda_{s}(G_{m})$ commutes with $SL(s-p)\times 1$, using (7.1), we get a homomorphism $f:SL(s-p)\times 1\longrightarrow H$.
		Since the character group of $SL(s-p)\times 1$ is trivial, $f$ is trivial. Therefore, $SL(s-p)\times 1$ fixes $x$. This is a contradiction to Step 1. Similarly, we can  prove that $V_{2}$ is non trivial.\\ Since both $V_{1}$ and $V_{2}$ are non trivial, $dim(V_{1}\otimes V_{2})\geq dim (V(\omega_{1})\otimes V(\omega_{1}))$. Further, equality holds only when $V_{1}=V(\omega_{1})$ or $V_{1}=V(\omega_{1})^{*}$ and  $V_{2}=V(\omega_{1})$ or $V_{2}=V(\omega_{1})^{*}$.
		
		On the other hand, since $dim(V_{1}\otimes V_{2}))=(s-p)(r-p)=k$, we have $V_{1}\otimes V_{2}=H^{0}(\mathbb{P}^{k-1},\mathbb{O}(1))$.	
		Further,	$p_{(1,2,\ldots,p,s,s+2,s+3,\ldots,s+r-p)}^{a}p_{(1,2,\ldots,p,s+1,s+2,\ldots,s+r-p)}^{m-a}$ is an $n\cdot \lambda_{s}$-invariant section of $\mathcal{L}(m\omega_{r})$. Further $T_{1}\times T_{2}$-weight of the above section is $(-w_{0,s}(a\omega_{1}), a\omega_{1})$, where $w_{0,s}$ is the longest element of Weyl group of $SL(s-p,\mathbb{C})$ relative to $T_{1}$.\\
		Therefore, $(-w_{0,s}(\omega_{1}),\omega_{1})$ is a weight of $H^{0}(\mathbb{P}^{k-1},\mathbb{O}(1))$. Therefore, $(\omega_{1},-w_{0,r}(\omega_{1}))$ is a weight of $H^{0}(\mathbb{P}^{k-1},\mathbb{O}(1))^{*}$. Hence, $V_{1}=V(\omega_{1})$ and $V_{2}=V(-w_{0,r}(\omega_{1}))$, where $w_{0,r}$ is the longest element of Weyl group of $SL(r-p)$ relative to $T_{2}$ .	\\
		Therefore, $H^{0}(\mathbb{P}^{k-1},\mathbb{O}(1))^{*}=V(\omega_{1})\otimes V(-w_{0,r}(\omega_{1}))$ as $SL(s-p)\times SL(r-p)$-module. This completes the proof.
	\end{proof}
	\begin{lemma}	
		Let $d\in \mathbb{N}$,  and $(m_1,\ldots,m_{n-1})\in \mathbb{Z}_{\geq0}^{n-1}$. Let 	$\mu=\displaystyle\sum_{i=1}^{n-1} m_{i}\omega_{i}$. Then $\mu\leq d\omega_{1}$ if and only if  there is a non negative integer $m_{n}$ such that $\displaystyle\sum_{i=1}^{n} m_{i} i=d$, where $\omega_{1}$ is the first fundamental weight of $SL(n,\mathbb{C})$.
		\label{weight An}
	\end{lemma}
	\begin{proof}(:$\Rightarrow$)
		Assume that $\mu\leq\ d\omega_{1}$. If $\mu=d\omega_{1}$, then take $m_{n}=0$, and the result holds. Now let $\mu < d\omega_{1}$, then there exists some positive root $\beta$ such that $\mu < \mu+\beta \leq d\omega_{1}$, and $\mu+\beta$ is dominant (see \cite[Corollary 2.7, p.349]{Stem 2}). We have $\beta= \alpha_{i}+\alpha_{i+1}+\cdots+\alpha_{j}$, for some $1\leq i\leq j < n$. Note that $\beta=-(\omega_{i-1}+\omega_{j+1})+(\omega_{i}+\omega_{j})$. We will prove the result by induction on height of $d \omega_{1}-\mu$.\\ $\mu+\beta=m_{1}\omega_{1}+..+(m_{i-1}-1)\omega_{i-1}+(m_{i}+1)\omega_{i}+\cdots+(m_{j}+1)\omega_{j}+(m_{j+1}-1)\omega_{j+1}+..+m_{n-1}\omega_{n-1}=\displaystyle\sum_{i=1}^{n-1} m^\prime_{i}\omega_{i}$ (say). Then, we have $\displaystyle\sum_{i=1}^{n-1} m^\prime_{i}i=\displaystyle\sum_{i=1}^{n-1} m_{i}i$.\\
		But $ht( d \omega_{1}-(\mu+\beta))< ht(d\omega_{1}-\mu)$ and by induction there exists $m^\prime_{n}\geq 0$ such that $\displaystyle\sum_{i=1}^{n} m^\prime_{i}i=d$. Hence, we have $\displaystyle\sum_{i=1}^{n-1} m^\prime_{i}i$+$m^\prime_{n}n=d$. Thus, we have $\displaystyle\sum_{i=1}^{n-1} m_{i}i+m^\prime_{n}n=d$.\\
		($\Leftarrow:$)	To prove the converse, we first claim that $\omega_{r}\leq r \omega_{1}$ for all $r=1,2,\ldots,n-1.$ Indeed, $r \omega_{1}-\omega_{r}=r\epsilon_{1}-(\epsilon_{1}+\cdots+\epsilon_{r})$=$(\epsilon_{1}-\epsilon_{2})+\cdots+(\epsilon_{1}-\epsilon_{r})\geq 0$.\\
		Now $\displaystyle\sum_{i=1}^{n-1} m_{i} \omega_{i}\leq \displaystyle\sum_{i=1}^{n-1} m_{i}(i\omega_{1})+m_{n}(n\omega_{1})$=$(\displaystyle\sum_{i=1}^{n} m_{i}i)\omega_{1}$=$d \omega_{1}$, and hence $\mu\leq d \omega_{1}$.\\
	\end{proof}	
	Let $n\leq m$. Consider $M(n,m)$ as $GL(n,\mathbb{C}) \times GL(m,\mathbb{C})$-module, where the action is $(A,B)\phi=A\phi B^{-1}$, for $\phi\in M(n,m)$, $A\in GL(n,\mathbb{C})$ and $B\in GL(m,\mathbb{C})$. 	
	We recall definition of diagram from  (see \cite[section A, p.132]{Eis}). For a tableau $\Gamma$, we  denote by $|\Gamma|$  the Young diagram associated to $\Gamma$. We think of $\Gamma$ as a way of filling in the 	\noindent ``boxes of $|\Gamma|$" with numbers from $1$ and $n$. We recall some  more notations and definition from \cite{Eis} (see \cite[p.138, 146]{Eis}).\\
	A tableaux $\Gamma$ is called standard if rows of $\Gamma$'s are strictly increasing sequences and columns are non-decreasing sequences.
	If $\sigma$ is a diagram of at most $n$ columns then there are two special standard tableaux of shape $\sigma$ defined as follows.
	\begin{definition}
		\begin{enumerate}
			\item (Canonical) The $i^{th}$ row of the canonical tableau $C_{\sigma}$ is defined  to be $(1,2,\ldots,\sigma_{i})$, where $\sigma_{i}$ denotes the length of the $i^{th}$ row of the diagram $\sigma$.
			\item (Anticanonical) The $i^{th}$ row of the anticanonical tableau $\overline{C_{\sigma}}$ is defined to be $(n-\sigma_{i}+1,\ldots,n-1,n)$.
		\end{enumerate}
	\end{definition}
	\begin{definition}(Double tableaux)
		A double (standard) tableaux is a pair $(\Gamma_{1}|\Gamma_{2})$ of (standard) tableaux, with $|\Gamma_{1}|=|\Gamma_{2}|$. \end{definition}
	We use $(\Gamma_{1}|\Gamma_{2})$ to indicate products of minors of a matrix : Let $X=(X_{ij})$ be  an $n\times m$ matrix with $n\leq m$ and $|\Gamma_{1}|=|\Gamma_{2}|$ has at most $n$ columns. Let $|\Gamma|_{i}$, denotes the length of the $i^{th}$ row of $|\Gamma|$. Then we associate to $(\Gamma_{1}|\Gamma_{2})$ the product of minors of X whose $i^{th}$ factor is the minor involving rows $\Gamma_{1}(i,1), \Gamma_{1}(i,2),\ldots,\Gamma_{1}(i,|\Gamma_{1}|_{i})$ and columns $\Gamma_{2}(i,1), \Gamma_{2}(i,2),\ldots,\Gamma_{2}(i,|\Gamma_{2}|_{i})$. Thus the $i^{th}$ factor is a minor of order $|\Gamma|_{i}$.\\
	Given tableaux $\Gamma_{1}$ and $\Gamma_{2}$ of the same shape, we write $(\Gamma_{1}|\Gamma_{2})$ for the double tableaux formed form $\Gamma_{1}$ and $\Gamma_{2}$. We say a double tableaux is right  (respectively, left) semi-canonical if it has the form $(\Gamma_{1}|C_{\sigma})$ (respectively, $(C_{\sigma}|\Gamma_{2})$. Let $L_{\sigma}$ and ${_\sigma}L$ be the spaces spanned respectively by all the right and left semi-canonical tableaux of shape $\sigma$.\\
	We recall the following results from \cite{Eis} (see \cite[Theorem 3.3, Corollary 3.4 and comment at the end of p.131]{Eis}).	
	\begin{theorem}
		\begin{enumerate}
			\item $L_{\sigma}$ is an  irreducible $GL(n,\mathbb{C})$-submodule of $\mathbb{C}[M(n,m)]$.
			\item $L_{\sigma}\otimes {_\sigma}L$ is an irreducible  $GL(n,\mathbb{C}) \times GL(m,\mathbb{C})$-module, and \[\mathbb{C}[M(n,m)]=\displaystyle\bigoplus_{\sigma\in\Gamma} (L_{\sigma}\otimes {_\sigma}L)\] where $\Gamma$ is the set of distinct diagrams $\sigma$ with at most $n$ columns.
		\end{enumerate}
	\end{theorem}
	\begin{proposition}
		For a fixed $d\in \mathbb{N}$, $Sym^{d}(M(n,m))^{*}=\displaystyle\bigoplus_{\sigma\in\Gamma_{d}} (L_{\sigma}\otimes {_\sigma}L)$, where $\Gamma_{d}$ is the set of all diagrams $\sigma$ with atmost $n$ columns such that $\displaystyle\sum_{i} \sigma_{i}=d$.\label{decom symm}
	\end{proposition}
	
	\subsection{The decomposition of homogeneous coordinate ring}
	\ 
	
	Let $R_{k}:= H^{0}(X(w_{s,r}),\mathcal{L}(km\omega_r))^{n\cdot\lambda_{s}}$. 
	We  decompose $R_{k}$ into irreducible $G^{\prime}$-modules.  We give the description of $R_{k}$ into irreducible $G^{\prime}$-modules in terms of the dominant weights relative to $(B_{1}\times B_{2},T_{1}\times T_{2})$.
	
	We first treat the case $s\leq r$.
	
	Let $\mathcal{D}_{ka}$ denotes the set of all diagrams $\sigma=(\sigma_{1}\geq \sigma_{2}\geq \cdots\geq\sigma_{q}) (q=length(\sigma))$ such that $\sigma_{1}\leq s-p$ and $\sum_{i=1}^{q}\sigma_{i}=ka$.
	
	Let $\{\omega_{i}:1\leq i\leq s-p-1)\}$ (resp. $\{\omega_{i}^{\prime}: 1\leq i\leq r-p-1\})$ denotes the fundamental weights of $SL(s-p)$ (resp. of $SL(r-p)$).
	For each $\sigma\in\mathcal{D}_{ka}$,  let $\mu_{\sigma}$ denote the dominant weight corresponding to $\sigma$ considered as weight for $SL(s-p)$.
	\begin{lemma}
		Let $\mu$ be a dominant weight of $T_{1}$. Then $\mu\leq ka\omega_{1}$ if and only if there exists unique $\sigma\in \mathcal{D}_{ka}$ such that $\mu_{\sigma}=\mu$.
		\label{type An weight inequality}
	\end{lemma}
	\begin{proof}
		($\Leftarrow:$) Let $\sigma\in \mathcal{D}_{ka}$ be such that $\mu_{\sigma}=\mu$. Then we have, $\sum_{j} \sigma_{j}=ka$ and $\sigma_{1}\leq s-p$. For  $1\leq i\leq s-p$, let $m_{i}$=$|\{j: \sigma_{j}=i\}|$. Then 
		$\mu_{\sigma}$ = $\displaystyle\sum_{i=1}^{s-p-1}m_{i}\omega_{i}$ is a  dominant weight of $T_{1}$. Also we have,  $\displaystyle\sum_{i=1}^{s-p}m_{i}\cdot i=ka$. Therefore, from Lemma \ref{weight An}, we have $\mu_{\sigma}\leq ka\omega_{1}$.\\
		($:\Rightarrow$) Conversely, let $\mu\leq ka\omega_{1}$ be a dominant weight of $T_{1}$. Let $\mu= \displaystyle\sum_{i=1}^{s-p-1}n_{i}\omega_{i}$, $n_{i}\geq 0$. Then by Lemma \ref{weight An}, there exists $n_{s-p}\geq 0$ such that $\displaystyle\sum_{i=1}^{s-p}n_{i}\cdot i=ka$. Consider the diagram $\sigma$ corresponding to the partition $((s-p)^{n_{s-p}},(s-p-1)^{n_{s-p-1}},\ldots,2^{n_{2}},1^{n_{1}})$ of $ka$. Then $\sigma\in \mathcal{D}_{ka}$ and $\mu_{\sigma}=\mu$.\\
		If $\sigma^{\prime}\in \mathcal{D}_{ka}$ be such that $\mu_{{\sigma}^{\prime}} =\mu$. Let $n_{i}^{\prime}$=$|\{j: \sigma_{j}^{\prime}=i\}|$. Then $ \mu_{{\sigma}^{\prime}}=\displaystyle\sum_{i=1}^{s-p-1}n_{i}^{\prime}\omega_{i}$. Since $\mu_{{\sigma}^{\prime}} =\mu_{\sigma}$, we have $n_{i}=n_{i}^{\prime} $ for all $1\leq i\leq s-p-1$. Also $\sum\sigma_{i}=\sum{\sigma_{i}}^{\prime}=ka$. This implies $(s-p)
		\cdot n_{s-p}=(s-p)\cdot n_{s-p}^{\prime}$. Since $p<s$, $n_{s-p}=n_{s-p}^{\prime}$. Hence, $\sigma=\sigma^{\prime}$.
	\end{proof}
	\begin{lemma} For a dominant weight $\mu$ of $T_{1}$, let V($\mu$) be the irreducible $SL(s-p)$-module with highest weight $\mu$. Let $\Bar{\mu}$  be the dominant weight of $T_{2}$ determined by $\mu$. Then $V(\mu)^{*}\otimes V(\Bar{\mu})$ is an irreducible $G^{\prime}$-module with highest weight $(-w_{0,s}(\mu),\bar{\mu})$, where $w_{0,s}$ is the longest element of the Weyl group of $SL(s-p)$. Then we have \\
		$R_{k}
		=\displaystyle\bigoplus_{\mu\leq ka\omega_{1}}(V(\mu)^{*}\otimes V(\Bar{\mu}))$, where $\omega_{1}$ is the first fundamental weight of $SL(s-p)$.
		\label{type An 1}
	\end{lemma}
	\begin{proof} From Theorem \ref{git quotient type An} we have, $R_{k}=Sym^{ka}(M(s-p,r-p))^{*}$. Using Proposition \ref{decom symm} for $n=s-p$ and $m=r-p$ we have, $Sym^{ka}(M(s-p,r-p))^{*}=\displaystyle\bigoplus_{\sigma\in \mathcal{D}_{ka}} (L_{\sigma}\otimes{ _{\sigma}L})$.
		
		Now in view of Lemma \ref{type An weight inequality},  it suffices to prove that, for $\sigma\in \mathcal{D}_{ka}$,  $L_{\sigma}\otimes {_\sigma}L$ is isomorphic to $V(\mu_{\sigma})^{*}\otimes V(\Bar{{\mu}_{\sigma}})$ considered as $G^{\prime}$-module.	From \cite[Theorem 3.3, p.147]{Eis},  $L_{\sigma}$ is an irreducible $GL(s-p,\mathbb{C})$-module and hence irreducible considering as $SL(s-p,\mathbb{C})$-module. Similarly, ${ _{\sigma}L}$ is an irreducible $SL(r-p,\mathbb{C})$-module. Further, $(\overline{C_{\sigma}}|C_{\sigma})$ (resp. $(C_{\sigma}|C_{\sigma})$) is a highest weight vector of $L_{\sigma} $ (resp. ${ _{\sigma}L}$) (see \cite[Corollary 3.4 (2), p.148]{Eis}). Now $T_{1}$-weight of $(\overline{C_{\sigma}}|C_{\sigma})$ is $-w_{0,s}(\mu_{\sigma})$ and $T_{2}$-weight of $(C_{\sigma}|C_{\sigma})$ is $\bar{\mu_{\sigma}}$ .\\
		This proves that $T_{1}\times T_{2}$-weight of $(\overline{C_{\sigma}}|C_{\sigma})\otimes(C_{\sigma}|C_{\sigma})$ is $(-w_{0,s}(\mu_{\sigma}),\bar{\mu_{\sigma}})$. Since $V(\mu)^{*}=V(-w_{0,s}(\mu))$ for any dominant character $\mu$ of $T_{1}$ (see \cite[Corollary 2.5, p.200]{Jan}), we have the desired conclusion.
	\end{proof}	
	
	\begin{remark}
		Let  $s>r$. Let $\mathcal{D}^{\prime}_{ka}$ denotes the set of diagrams $\sigma=(\sigma_{1}\geq\sigma_{2}\geq\cdots\geq\sigma_{q})(q=length(\sigma))$ such that $\sigma_{1}\leq r-p$ and $\sum_{i=1}^{q}\sigma_{i}=ka$. Then by Proposition \ref{decom symm}, 
		$Sym^{ka}(M(s-p,r-p))^{*}=\displaystyle\bigoplus_{\sigma\in \mathcal{D}^{\prime}_{ka}} (L_{\sigma}\otimes {_{\sigma}L})$. A similar argument as in Lemma \ref{type An weight inequality} proves that, if  $\mu$ is a dominant weight of $T_{2}$, then $\mu\leq ka\omega_{1}^{\prime}$ if and only if there is a unique $\sigma\in \mathcal{D}^{\prime}_{ka}$ and $\mu_{\sigma}=\mu$. Then computation of highest weight proves that 
		$R_{k}=\displaystyle\bigoplus_{\mu\leq ka\omega_{1}^{\prime}}(V(\bar{\mu})^{*}\otimes V(\mu))$, where $\mu$ is dominant weight of $T_{2}$ and $\bar{\mu}$ is determined by $\mu$.
			\end{remark}
	\begin{corollary}
		For $s=r$, we have $H^{0}(X(w_{r,r}),\mathcal{L}(km\omega_r))^{n\cdot\lambda_{r}}=H^{0}(\overline{PSL(r-p)},\mathcal{L}(ka\omega_{1}))$.
	\end{corollary}
	\begin{proof}
		From Lemma \ref{type An 1}, $R_{k}=\displaystyle\bigoplus_{\mu\leq ka\omega_{1}}(V(\mu)^{*}\otimes V(\mu))=\displaystyle\bigoplus_{\mu\leq ka\omega_{1}}End(V(\mu)^{*})$. By part 1 of  \cite[Theorem 8.3, p.30]{De processi} for the special case of the wonderful compactification $\overline{PSL(r-p)}$ of $PSL(r-p)$, we have  $H^{0}(X(w_{r,r}),\mathcal{L}(km\omega_r))^{n\cdot\lambda_{r}}=H^{0}(\overline{PSL(r-p)},\mathcal{L}(ka\omega_{1}))$.
	\end{proof}
	\subsection{Acknowledgement}
	We thank Parameswaran Sankaran for useful discussions. We thank Infosys foundation for partial financial support. The first named author thank the NBHM for financial support.
		
\end{document}